\documentclass[11pt,a4paper]{article}
\usepackage[margin=1.3in]{geometry}
\usepackage[utf8]{inputenc}
\usepackage[T1]{fontenc}
\usepackage{tikz,tkz-tab}
\usepackage{fancyhdr}
\usepackage{tkz-euclide}
\usepackage[dvipsnames]{xcolor}
\usepackage{framed}

\usepackage{mathtools}

\usepackage{enumerate}

\usepackage{booktabs}

\usepackage{tablefootnote}

\usepackage{float}

\usepackage[colorlinks, pagebackref=true]{hyperref}
\hypersetup{
  linkcolor=[rgb]{0.3,0.3,0.6},
  citecolor=[rgb]{0.2, 0.6, 0.2},
  urlcolor=[rgb]{0.6, 0.2, 0.2}
}
\definecolor{darkblue}{rgb}{0.2,0.2,0.6} %
\definecolor{mediumblue}{rgb}{0.3,0.5,0.8} %
\definecolor{lightblue}{rgb}{0.6,0.8,1} %

\usepackage{mathtools, amsthm, amsfonts, amssymb}
\usepackage{tikz}
\usetikzlibrary{decorations.markings}
\usetikzlibrary {arrows.meta}

\usepackage{authblk}

\usepackage{listings}

\usepackage{accents} %
\usepackage{microtype}
\usepackage{xcolor}
	\definecolor{Azure4}{rgb}{0.33, 0.33, 0.33}
  
    \definecolor{Azure2}{rgb}{0.57, 0.64, 0.69}

\usepackage{thmtools, thm-restate} 
\declaretheorem[name=Theorem, parent=section]{theorem}
\declaretheorem[name=Corollary, sibling=theorem]{corollary}
\declaretheorem[name=Proposition, sibling=theorem]{proposition}
\declaretheorem[name=Lemma, sibling=theorem]{lemma}

\theoremstyle{definition}

\declaretheorem[name=Remark, sibling=theorem]{remark}

\theoremstyle{remark}

\newcommand{\RR}{\mathbb{R}}
\newcommand{\QQ}{\mathbb{Q}}
\newcommand{\NN}{\mathbb{N}}
\newcommand{\ZZ}{\mathbb{Z}}

\newcommand{\asympleq}{\lesssim_{\mathcal{G}}}

\newcommand{\homomoleq }{\leq_{\mathcal{G}}}

\let\eps\varepsilon

\DeclareMathAccent{\wtilde}{\mathord}{largesymbols}{"65}

\definecolor{deepcarrotorange}{rgb}
{0.605, 0.26, 0.12}

\definecolor{lightcarrotorange}{rgb}{0.91, 0.51, 0.27}

\title{Universality of asymptotic graph homomorphism}
\author[1]{Anna Luchnikov}
\author[2]{Jim Wittebol}
\author[1]{Jeroen Zuiddam}
\date{\today}

\affil[1]{University of Amsterdam, Amsterdam, Netherlands}
\affil[2]{Vrije Universiteit Brussel, Brussels, Belgium}

\begin{document}

\maketitle

\small
\centerline{\textbf{Abstract}}
\vspace{0.5em}
The Shannon capacity of graphs, introduced by Shannon in 1956 to model zero-error communication, asks for determining the rate of growth of independent sets in strong powers of graphs. Despite much interest, which has led to a large collection of upper bound methods (e.g., Lovász theta function, complement of the projective rank, fractional Haemers bound) and lower bound constructions, much is still unknown about this parameter, for instance whether it is computable. Indeed, results of Alon and Lubetzky (2006) have ruled out several natural routes to such an algorithm.

Recent work has established a dual characterization of the Shannon capacity in terms of the asymptotic spectrum of graphs (a class of well-behaved graph parameters that includes the aforementioned upper bound methods). A core step in this duality theory is to shift focus from Shannon capacity itself to studying the asymptotic relations between graphs, that is, the asymptotic cohomomorphisms: given graphs $G$ and $H$, is there a cohomomorphism (a map on vertex sets, mapping non-edges to non-edges) from the $n$th power of~$G$ to the $(n + o(n))$th strong power of $H$? Indeed, Shannon capacity essentially reduces to the case that $G$ has no edges (i.e.,~is an independent set).

Towards understanding the structure of Shannon capacity, we study the ``combinatorial complexity'' of asymptotic cohomomorphism. As our main result, we prove that the asymptotic cohomomorphism order is \emph{universal} for all countable preorders. That is, we prove that any countable preorder can be order-embedded into the asymptotic cohomomorphism order (i.e.~appears as a suborder). Previously this was known for (non-asymptotic) cohomomorphism (going back to work of Hedrlín (1969) and Pultr--Trnková (1980)), which left open the possibility that asymptotic cohomomorphism has a simpler structure.

The main strategy of our proof is to construct an order-embedding from an order on certain sets of finite binary strings, which is known to be universal by a result of Hubička and Nešetřil (2005), to the asymptotic cohomomorphism order. The construction of this embedding relies on results of Vrana (2021) on the convex structure of the asymptotic spectrum of graphs, and a new result determining the value of the complement of the projective rank on a class of circulant graphs (fraction graphs).  
These ingredients allow us to simulate a certain set of lines with pointwise order as graphs under asymptotic cohomomorphism, which we use to prove our result.
\newpage\normalsize

\section{Introduction}

Motivated by the Shannon capacity problem and recent approaches to it, we study in this paper a preorder on graphs called asymptotic cohomomorphism, which essentially measures if large powers of a graph allow a cohomomorphism into slightly larger powers of another graph. The homomorphism order (and thus also the cohomomorphism order) has been long known to be  very rich in structure. Indeed it goes back to work of Hedrlín \cite{Hedrln1969OnUP,Pultr1980CombinatorialAA} (see also Hell--Nešetřil \cite{Hell2004GraphsAH}) that any countable preorder can be order-embedded into it---a property called \emph{universality}---which was followed by results proving even more structure later by Fiala, Hubička, Long and Nešetřil \cite{FractalProp}.  Our main result is that asymptotic cohomomorphism, while a much ``stronger'' preorder than cohomomorphism (i.e.~with more relations, which could a priori greatly simplify its structure), is still universal. 

Posed by Shannon in 1956 \cite{MR0089131}, the Shannon capacity problem asks to determine the rate of growth of the independence number of a graph under taking strong graph product powers. This problem is notoriously difficult and has defied many attempts at solving it, even for specific small graphs, and despite much effort. Indeed, a long line of work has led to many methods and results, in the direction of upper bounds (Lovász theta function \cite{Lovasz79}, fractional Haemers bound \cite{haemers1979some,bukh2018fractionalversionhaemersbound}, complement of the projective rank \cite{Man_inska_2016}), lower bounds (e.g., Baumert--McEliece--Rodemich--Rumsey--Stanley--Taylor~\cite{MR0337668}, Polak--Schrijver \cite{MR3906144}, de Boer--Buys--Zuiddam \cite{deboer2024asymptoticspectrumdistancegraph}), and structural results (e.g., Alon \cite{MR2234473}, Alon--Lubetzky \cite{MR2234473}, Zuiddam \cite{SpectrumGraphs}, Vrana \cite{vrana2019probabilisticrefinementasymptoticspectrum}, Schrijver \cite{MR4525548},  Wigderson--Zuiddam \cite{wigderson2022asymptotic}). 

Many questions about Shannon capacity are open. One central open problem is whether Shannon capacity is computable. Alon--Lubetzky \cite{MR2234473} ruled out several natural approaches to constructing such an algorithm by showing that the ``jump'' behaviour of the independence number under powers can be very intricate. On the other hand, there are many examples, in several settings in mathematics and computer science, where parameters that are asymptotic or amortized allow a surprisingly simple description. The work in this paper is aimed at a better understanding of the ``complexity'' of Shannon capacity from a combinatorial point of view. 

For this we focus on asymptotic cohomomorphism, an order on graphs that is tightly related to Shannon capacity (essentially replacing the question of embedding large independent sets in powers of a graphs, by embedding large powers of a graph into large powers of another graph) and that was recently introduced in the development of asymptotic spectrum duality \cite{SpectrumGraphs} (see also~\cite{wigderson2022asymptotic}), and  plays a central role there. Asymptotic spectrum duality gives a dual characterization of Shannon capacity as a minimization problem over a class of functions called the asymptotic spectrum of graphs (which includes aforementioned Lovász theta function, fractional Haemers bound and complement of the projective rank, and also fractional clique covering number), but moreover also characterizes asymptotic cohomomorphism. 

We summarize here our main results, which we expand on in the rest of the text:
\begin{itemize}
    \item Motivated by the study of Shannon capacity, we prove that asymptotic graph cohomomorphism is universal for all countable preorders. In other words, every countable preorder appears as a suborder in it. In fact, our construction also implies universality of the (non-asymptotic) cohomomorphism order, thus giving a different proof for that.
    \item Our proof of universality relies on embedding a preorder on sets of binary strings (that was proven to be universal by Hubička and Nešetřil \cite{HubickaNestrilUniversal}) into graphs. It further relies on the asymptotic spectrum of graphs \cite{SpectrumGraphs}, its (convex) structure \cite{vrana2019probabilisticrefinementasymptoticspectrum}, and a construction of lines and points with special order properties.
    \item As an important ingredient for our proof, we prove that the complement of the projective rank coincides with the fractional clique covering number on a certain set of circulant graphs (fraction graphs), which allows us to ``simulate'' rational numbers. (See  \autoref{sec:fhb}.)
    
\end{itemize}

\section{Graph cohomomorphism and universality}

Let $\mathcal{G}$ be the class of all finite simple graphs. (In this paper we will only consider such graphs.) For any graph $G$, we denote by $V(G)$ its vertex set and by $E(G)$ its edge set. For any subset $S \subseteq V(G)$, $G[S]$ refers to the subgraph of $G$ induced by $S$. We denote by $G\boxtimes H$ the strong product of graphs, by $G \sqcup H$ the disjoint union of graphs, by $\overline{G}$ the complement graph, by $G + H = \overline{\overline{G} \sqcup \overline{H}}$ the join of graphs, by $K_n$ the complete graph on $n$ vertices, and by $E_n = \overline{K}_n$ its complement.

For any graphs $G,H$, a cohomomorphism $\phi : G \to H$ is a map $\phi: V(G) \to V(H)$ that maps distinct non-adjacent vertices to distinct non-adjacent vertices. In other words, $\phi : G \to H$ is a cohomomorphism if and only if $\phi : \overline{G} \to \overline{H}$ is a graph homomorphism. We write $G \leq_{\mathcal{G}} H$ if there is a graph cohomomorphism $\phi : G \to H$ and we call this relation $\leq_{\mathcal{G}}$ on $\mathcal{G}$ \emph{cohomomorphism preorder}. The cohomomorphism preorder is  well-studied (but usually in the complementary language of homomorphism), and much is known about its rich structure. One such structural property is universality.

Given two preorders $(A, \leq_A)$ and $(B, \leq_B)$ we call a map $f : A \to B$ an \emph{order-embedding} if for every $a_1, a_2 \in A$ we have $a_1 \leq_A a_2$ if and only if $f(a_1) \leq_B f(a_2)$. We call $(B, \leq_B)$ \emph{countably universal} if for every countable preorder $(A, \leq_A)$ there exists an order-embedding $f : A \to B$. In a preorder $(A, \leq_A)$, a subset $S \subseteq A$ is called an \emph{antichain} if for every distinct $a,b \in A$, $a \nleq_A b$ and $b \nleq_A a$.

The (abstract) existence of countably preorders has been known since the work of Fraissé \cite{FraisseUniversality}. It was established by Hedrlín, Pultr and Trnková \cite{Hedrln1969OnUP, Pultr1980CombinatorialAA} that the cohomomorphism preorder on graphs is countably universal. Different proofs were later given by Hubička and Nešetřil \cite{FInitePathsAreUniversal, HUBICKA2005765}, and the universality was strengthened to a ``fractal'' property by Fiala, Hubička, Long and Nešetřil \cite{FractalProp}.

\section{Result: universality of asymptotic cohomomorphism}

For any graph $G$, the Shannon capacity of $G$ is defined as $$\Theta(G) = \sup_{n} \alpha(G^{\boxtimes n})^{1/n},$$ where~$\alpha$ denotes the independence number, and where the supremum may be equivalently replaced by a limit (by Fekete's lemma). 
Shannon capacity is closely related to a preorder on graphs called the \emph{asymptotic cohomomorphism preorder} introduced in \cite{SpectrumGraphs} (see also \cite{Fritz2017}). For graphs $G$, $H$, we write $H \asympleq G$ if and only if there exists a function $f: \mathbb{N} \to \mathbb{N}$ satisfying $\lim_{n \to \infty} f(n)/n = 0$, such that for all $n$, we have $$
    H^{\boxtimes n} \homomoleq G^{\boxtimes(n + f(n))}.$$
(We will denote by $o(n)$ any function $f: \mathbb{N} \to \mathbb{N}$ for which $\lim_{n \to \infty} f(n)/n = 0$, so that the above may be written as $H^{\boxtimes n} \homomoleq G^{\boxtimes (n + o(n))}$.) Cohomomorphism implies asymptotic cohomomorphism, but not necessarily the other way around. For instance $E_{11} \nleq_{\mathcal{G}}C_5^{\boxtimes 3} $ but $E_{11} \asympleq C_5^{\boxtimes 3}$. (Indeed, it is known that $\alpha(C_5^{\boxtimes 3}) = 10 $ \cite{BadMna12}, whence $E_{11} \nleq_{\mathcal{G}} C_5^{\boxtimes 3}. $ On the other hand, one can check with the definition of the Shannon capacity that $\Theta(C_5^{\boxtimes 3})= \Theta(C_5)^3= 5^{3/2}>11 $ \cite{MR0089131}, implying  $\alpha(C_5^{\boxtimes 3n }) \geq 11^{n -o(n)}$ and so $\alpha(C_5^{\boxtimes (3n +o(n))}) \geq 11^{n }$.  Therefore, $C_5^{\boxtimes(3 n + o(n))}\geq E_{11^n} = E_{11}^{\boxtimes n} $. Hence  $E_{11}\asympleq C_5^{\boxtimes 3}$.)

In the context of the study of  asymptotic cohomorphism, natural questions arise: How complex is the asymptotic cohomomorphism preorder? Could it have a far simpler structure than the cohomomorphism preorder? Can it be determined by looking at only a few properties of the graphs $G$ and~$H$? Our main result asserts that, in a precise sense, the asymptotic cohomomorphism preorder is as complex as the cohomomorphism order.

Let $(\mathcal{G}, \asympleq)$ denote the set of (finite simple) graphs equipped with the asymptotic cohomomorphism preorder. 

\begin{theorem}\label{th:main}
    $(\mathcal{G}, \asympleq)$ is countably universal.
\end{theorem}

The high-level steps of the proof of \autoref{th:main} are as follows:

\begin{enumerate}[(1)]

\item The starting point is a preorder $(W, \leq_W)$ on binary strings and a preorder $(\mathcal{W},\leq_{\mathcal{W}})$ on antichains of elements in $W$. The latter is known to be countably universal~\cite{HubickaNestrilUniversal}. (See \autoref{sec:ups} for the definition of these preorders.)

\item We embed $W$ into the set of graphs~$\mathcal{G}$ through a construction of lines and points, assigning to every $a \in W$ a graph~$G_a$. 
To ensure that this embedding is order-preserving, we  use tools from asymptotic spectrum duality (in particular convexity) and special knowledge of some elements in the asymptotic spectrum (fractional clique covering number and complement of the  projective rank).

\item We then embed $\mathcal{W}$ into the set of graphs~$\mathcal{G}$, assigning to every $A \in \mathcal{W}$ the graph $\sum_{a \in A} G_a$, where sum means graph join. The way we set up the embedding in (2) will then allow us to prove that this embedding is order-preserving.
\end{enumerate}

The rest of the paper is organized as follows.
In~\autoref{sec:asd} and~\autoref{sec:asd-vtg} we give the necessary ingredients from asymptotic spectrum duality  theory. A new result on the complement of the projective rank, which is also of independent interest, is presented in~\autoref{sec:fhb}. The latter implies the existence of a special family of spectral points that behave nicely with respect to certain graphs. We derive the existence of such a family in~\autoref{sec:fam-spec}. After that, as a warm-up, we illustrate how to use our techniques to order-embed some special preorders into $(\mathcal{G}, \asympleq)$, namely any countable antichain and any finite preorder (\autoref{sec:sbr}). The main result (\autoref{th:main}) is proved in~\autoref{sec:ups}. Finally, we discuss some open problems in~\autoref{sec:dis}.

\section{Asymptotic spectrum of graphs}\label{sec:asd}

An essential tool to study the asymptotic cohomomorphism preorder $(\mathcal{G}, \asympleq)$, is the asymptotic spectrum of graphs, introduced in \cite{SpectrumGraphs}. The theory of asymptotic spectra was originally developed by Strassen  \cite{strassen1987relative, strassen1988asymptotic, strassen1991degeneration} to study tensors and tensor rank, and more generally the asymptotic ``rank'' of elements in the  setting of partially ordered well-behaved semirings.
The asymptotic spectrum of graphs, denoted by $\Delta(\mathcal{G})$, is defined as the set of functions $f: \mathcal{G} \to \RR_{\geq 0}$ that satisfy the following four properties, for every $n \in \NN$ and every two graphs $G,H$:
\begin{enumerate}[(i)]
     \item %
     $f(E_n) = n$.
    
    \item  %
    $f(G \boxtimes H) = f(G)f(H).$
    
    \item  %
    $f(G \sqcup H) = f(G) + f(H).$

    \item  %
    If $G \homomoleq H$, then $f(G) \leq f(H).$
 \end{enumerate}
The functions in $\Delta(G)$ are called \emph{spectral points}. They characterize the preorder $\asympleq$ completely:
 \begin{theorem}[\cite{SpectrumGraphs}] Let $G,H$ be graphs. Then
 $G \asympleq H $ if and only if for every $f \in \Delta(\mathcal{G})$, $f(G) \leq f(H).$
     
 \end{theorem} \label{DualityThm}
In addition to the four defining properties of the functions belonging to the asymptotic spectrum, spectral points behave well under the join operation $+$.
 \begin{lemma}[\cite{vrana2019probabilisticrefinementasymptoticspectrum}] \label{SumOfGraphs}
 Let $G,H$ be graphs. Then
$f(G + H) = \max(f(G),f(H))$ for every $f\in \Delta(\mathcal{G})$.
\end{lemma}

\section{The asymptotic spectrum and vertex-transitive graphs} \label{sec:asd-vtg}
For our construction, we will use a property of the asymptotic spectrum that guarantees the existence of spectral points taking all values between two given spectral point values.

\begin{lemma}\label{interpol}
    Let $f_0,f_1\in \Delta(\mathcal{G})$. %
    For every $0\leq \lambda \leq 1$, there exists $f_\lambda \in \Delta(\mathcal{G})$ such that for every vertex-transitive graph $G$, we have $f_\lambda(G) = f_0(G)^{1-\lambda}f_1(G)^{\lambda}$. 
    
\end{lemma}

 To prove  \autoref{interpol}, we will use  machinery introduced by Vrana~\cite{vrana2019probabilisticrefinementasymptoticspectrum}.
Let~$G$ be a non-empty graph. We denote by $\mathcal{P}(G)$ the space of probability distributions on $V(G)$. %
For any $P \in \mathcal{P}(G)$, $n \in \NN$ and $\eps > 0$, let $T_{B_\eps(P)}^n \subseteq V(G)^n$ be the set of $n$-tuples whose empirical distribution is $\eps$-close to $P$ in $\ell_1$-norm. 
For each spectral point $f \in \Delta(\mathcal{G})$ and probability distribution $P \in \mathcal{P}(G)$, the  
\emph{probabilistic refinement} of $f$ is defined as 
$$f(G,P) = \inf_{\eps > 0} \limsup_{n\to0}  \sqrt[n]{f(G^{\boxtimes n}[T^n_{B_\eps(P)}])}.$$ We state  an important direct consequence of \cite[Theorem~1.1 and Theorem~1.2]{vrana2019probabilisticrefinementasymptoticspectrum}.

\begin{theorem} \label{thmVrana} For every $f \in \Delta(\mathcal{G})$ and every non-empty graph $G$,
\begin{equation} \label{eq:1}
    f(G) = \max_{P \in \mathcal{P}(G)}f(G,P).
\end{equation}
Furthermore, for every $f_0,f_1 \in \Delta(\mathcal{G})$ and every $0\leq \lambda \leq 1$, the function $$ (G,P) \mapsto  f_0(G,P)^{1-\lambda }f_1(G,P)^{\lambda}$$
is the %
probabilistic refinement of some $f_\lambda \in \Delta(\mathcal{G})$, and $$f_\lambda(G) = \max_{P \in \mathcal{P}(G)} f_0(G,P)^{1-\lambda} f_1(G,P)^{\lambda}.$$
\end{theorem}
Two other properties which will be useful to us are:

\begin{lemma}[{\cite{vrana2019probabilisticrefinementasymptoticspectrum}, see also \cite[Lemma 2.15]{deboer2024asymptoticspectrumdistancegraph}}]\label{VtransitiveGraphs} Let $G$ be a vertex-transitive graph and $S \subseteq V(G)$.  
Then for all $f \in \Delta(\mathcal{G})$, we have $$f(G) \leq \frac{|V(G)|}{|S|} f(G[S]).$$
\end{lemma}

\begin{lemma}[{\cite[Proposition 3.4]{vrana2019probabilisticrefinementasymptoticspectrum}}]\label{lem:cong-seq}
Let $G$ be a graph and $P \in \mathcal{P}(G)$. Suppose that $(P_n)_{n \in \mathbb{N}} \subseteq \mathcal{P}(G)$ is a sequence converging to $P$ in $\ell_1$-norm such that for all $n \in \mathbb{N}$ and all $v \in V(G)$ we have $nP_n(v) \in \mathbb{N}$.  %
Then 
$$f(G,P) = \lim_{n \to \infty} \sqrt[n]{f(G^{\boxtimes n}[T_{P_n}^n])}. $$
\end{lemma}
In particular, \autoref{lem:cong-seq} implies the following.

\begin{lemma}[\cite{vrana2019probabilisticrefinementasymptoticspectrum}]\label{rationalPoints}
Let $n \in \mathbb{N}$ and $P \in \mathcal{P}(G)$ be a probability distribution such that $nP(v) \in \NN$ for all $v \in V(G)$. We can define $T^{kn}_P \subseteq V(G)^{kn}$ to be the set of $kn$-tuples in which every $v \in V(G)$ appears exactly $knP(v)$ times. Then
\[
f(G,P) = \lim_{k \to \infty} \sqrt[kn]{f(G^{\boxtimes kn}[T^{kn}_P])}.
\]
\end{lemma}

 The following fact is present in \cite{vrana2019probabilisticrefinementasymptoticspectrum} albeit not spelled out in a separate lemma. For convenience of the reader, we give it with a proof here. 

\begin{lemma}\label{vTransProp}
    Let $G$ be a vertex-transitive graph, $U  \in \mathcal{P}(G)$ the uniform distribution on its vertices and $f \in \Delta(\mathcal{G})$. Then $f(G) = f(G,U)$. 
\end{lemma}

In other words, \autoref{vTransProp} says that for vertex-transitive graphs $G$ the maximum in \autoref{eq:1} is achieved at $P = U$.

\begin{proof}

    Let $k= |V(G)|$. For all $n$, the graph $G^{\boxtimes kn}[T_U^{kn}]$ is an induced subgraph of the vertex-transitive graph~$G^{\boxtimes kn}$. Furthermore, the size of  $T_U^{kn}$ is bounded from below by $$|T_U^{kn}| \geq \frac{1}{(kn+1)^{|V(G)|}}2^{knH(U)} = \frac{1}{(kn+1)^{|V(G)|}}|V(G)|^{kn}.$$ 
    Therefore, by \autoref{VtransitiveGraphs}, for any $ f \in \Delta(\mathcal{G})$,
    $$f(G^{\boxtimes{kn}}) \leq \frac{|V(G^{\boxtimes kn})|}{|V(G^{\boxtimes kn}[T_U^{kn}])|}  f(G^{\boxtimes kn}[T_U^{kn}]) \leq (kn+1)^{|V(G)|} f(G^{\boxtimes kn}[T_U^{kn}]).$$
    Recall that by properties of $\Delta(\mathcal{G})$, $f(G^{\boxtimes kn}) =f(G)^{kn}$. Thus, taking $kn$th roots and letting $n$ go to $\infty$, this gives us that  
    \[
    f(G) \leq \lim_{n \to \infty} \sqrt[kn]{f(G^{\boxtimes kn}[T_U^{kn}])} = f(G,U),
    \]
    where the last equality is true due to \autoref{rationalPoints}. The reverse inequality $f(G) \geq f(G,U)$ follows from \autoref{eq:1} in \autoref{thmVrana}.
\end{proof}

We are now in a position to give a proof of \autoref{interpol}.

\begin{proof}[Proof of \autoref{interpol}]
   
    By \autoref{thmVrana}, the function $f_\lambda $ given for non-empty graphs $G$ by $$f_\lambda(G) = \max_{P \in \mathcal{P}(G)} f_0(G,P)^{1-\lambda }f_1(G,P)^{\lambda}$$ is an element of $\Delta(\mathcal{G})$. 
    Let $G$ be vertex-transitive. Then the maximal value over $P\in\mathcal{P}(G)$ of both $f_0(G,P)$ and $f_1(G,P)$ is reached when $P = U$ by \autoref{vTransProp}. 
    Hence, 
    \begin{align*}
      f_\lambda(G) &= \max_{P \in \mathcal{P}(G)} f_0(G,P)^{1-\lambda} f_1(G,P)^{\lambda}\\
        &= f_0(G,U)^{1-\lambda} f_1(G,U)^{\lambda}
        = f_0(G)^{1-\lambda }f_1(G)^{\lambda}. \qedhere
    \end{align*}
    
\end{proof}

We see that we have more control over spectral points when we evaluate them on  vertex-transitive graphs. In this regard, we make a useful observation.
 \begin{lemma}\label{lem:vertexTrans-Example}
     For any distinct $f_0, f_1 \in \Delta(\mathcal{G})$, there exists a vertex-transitive graph $G$ such that $f_0(G)\neq f_1(G)$. 
 \end{lemma}
\begin{proof}
   
 Since $f_0 \neq f_1$, there exists some graph $H$ such that, say,  $f_0(H) >f_1(H)$. Thus, $\max_{P \in \mathcal{P}(H)} f_0(H,P) > \max_{P \in \mathcal{P}(H)} f_1(H,P)$. Let $Q \in \mathcal{P}(H)$ be such that the maximum on the left-hand side is reached. 
 Then $$f_0(H,Q) > \max_{P \in \mathcal{P}(H)} f_1(H,P) \geq f_1(H,Q).$$ By \autoref{lem:cong-seq}, \[
 f_0(H,Q) = \lim_{n \to \infty}  \sqrt[n]{f_0(H^{\boxtimes n}[T^n_{Q_n}]) } \quad \text{ and } \quad f_1(H,Q) = \lim_{n \to \infty}    \sqrt[n]{f_1(G^{\boxtimes n}[T^n_{Q_n}]) }
 \]
 for any sequence $Q_n \in \mathcal{P}(H)$ with  $nQ_n(v)\in \mathbb{N} \,\,\,\,\forall v \in V(H)$ that converges in $\ell_1$-norm to $Q$. Hence, for $n$ large enough,  $\sqrt[n]{f_0(H^{\boxtimes n}[T^n_{Q_n}]) }  >  \sqrt[n]{f_1(H^{\boxtimes n}[T^n_{Q_n}]) }$ and so $$f_0(H^{\boxtimes n}[T^n_{Q_n}]) > f_1(H^{\boxtimes n}[T^n_{Q_n}]).$$ The graph $ H^{\boxtimes n}[T^n_{Q_n}] $ is vertex-transitive, whence the claim.
\end{proof}

\section{Complement of projective rank equals fractional clique covering number on fraction graphs}
\label{sec:fhb}
We will use a special family of graphs, which we call fraction graphs, to ``simulate'' rational numbers in our construction of a universal suborder of $(\mathcal{G}, \asympleq)$.
These have previously been used to study variations of the chromatic number~\cite{Vince1988StarCN} and to construct new lower bounds on Shannon capacity of odd cycles \cite{MR3906144, deboer2024asymptoticspectrumdistancegraph}.\footnote{They appear in the literature also under the names of circular graphs \cite{polak2020newmethodscodingtheory}, cycle-powers, (the
complement of) rational complete graphs \cite{Hell2004GraphsAH} and circular complete graphs~\cite{zhuColourings}.}

For $p, q \in \mathbb{N}$ such that $p/q \geq 2$ we denote by $E_{p/q}$ the graph with vertex set $\{0,1, \dots, p-1\}$, which we identify with $\ZZ_p$, and so that two vertices $i,j$ are adjacent if and only if the distance between them is strictly less than $q$ $\pmod p$. These we call \emph{fraction graphs}. They model the rational numbers $\mathbb{Q}_{\geq 2}$ with their usual ordering
inside the  preorder $(\mathcal{G}, \asympleq)$, as follows:

\begin{lemma}[\cite{Hell2004GraphsAH, deboer2024asymptoticspectrumdistancegraph}]
For any $p,q,r,s \in \NN$ with $p/q, r/s \geq2$, the following are equivalent:
\begin{enumerate}[\upshape(i)]
    \item $E_{p/q} \homomoleq E_{r/s}$,
    \item $E_{p/q} \asympleq E_{r/s}$,
    \item $p/q\leq r/s$.
\end{enumerate}
\end{lemma}

In particular, $E_{p/q}$ and $E_{r/s}$ are equivalent under asymptotic cohomomorphism  if and only if they are equivalent under cohomomorphism if and only if $p/q = r/s$ if and only if $\phi(E_{p/q})= \phi(E_{r/s})$ for all $\phi \in \Delta(\mathcal{G}).$

The \emph{fractional clique covering number} $\overline{\chi_f}(G)$ of any graph $G$ is defined as the supremum over all fractions $a/b$ with $a, b \in \NN$ for which there exists an assignment of subsets of $\{1,\dots,a\}$ of size $b$ to each vertex of $G$ such that non-adjacent vertices receive disjoint sets. 
(For other characterizations see \cite{Scheinerman1997FractionalGT}.)
It is known that $\overline{\chi_f} \in \Delta(\mathcal{G})$ and that $\overline{\chi_f}(G) = \max_{f \in \Delta(\mathcal{G})} f(G)$~\cite{SpectrumGraphs}.
It is also known that   $\overline{\chi_f}$ is well behaved when evaluated on fraction graphs:

\begin{lemma}[{\cite[Proposition 3.2.2]{Scheinerman1997FractionalGT}}]\label{chilemma}
   For every $p,q\in \mathbb{N}$ such that $p/q \geq 2$, we have  $\overline{\chi_f}(E_{p/q}) = p/q.$
\end{lemma}
The \emph{complement of the projective rank} of a graph $G$, denoted by $\overline{\xi_{f}}(G)$, is defined as the infimum over all fractions $d/r$ with $d,r \in \NN $ for which there exists an assignment of subspaces $W_v \leq \mathbb{C}^{d}$ of dimension $r$ to each vertex $v$ of $G$ such that non-adjacent vertices are assigned orthogonal subspaces.
(This parameter is related to the projective rank $\xi_{f}(G)$, initially introduced in~\cite{Man_inska_2016},  via $\overline{\xi_f}(G) = \xi_f(\overline{G})$.) It is known that $\overline{\xi_f} \in \Delta(\mathcal{G})$ \cite{Man_inska_2016,SpectrumGraphs}.

We will prove the analogous statement to \autoref{chilemma} for $\overline{\xi_f}$:

\begin{lemma}\label{xilemma}
    For every $p,q\in \mathbb{N}$ such that $p/q \geq 2$, we have $\overline{\xi_f}(E_{p/q}) = p/q.$
\end{lemma}
Since for every graph $G$ we have $\overline{{\xi}_f}(G) \leq \overline{\chi_f}(G)$, and so in particular for every $p,q \in \NN$ we have
    $\overline{\xi_f}(E_{p/q}) \leq p/q$, it remains to prove $\overline{\xi_f}(E_{p/q}) \geq p/q$.

\begin{remark}It follows from \autoref{xilemma} that  $f_0=\overline{\xi_f}, \,\,f_1 =\overline{\chi_f}\in \Delta(\mathcal{G})$  are such that for all $p/q \in \QQ_{\geq 2}$, $f_0(E_{p/q}) = f_1(E_{p/q}) = p/q$. This will be crucial in the next section.\end{remark}

 Before we proceed to the proof of \autoref{xilemma}, we need some  auxiliary lemmas. But, first, note that by \cite[Lemma 6.13.1]{Roberson2013VariationsOA}, in the particular case of vertex-transitive graphs, one can equivalently use an alternative definition of complement of the projective rank. In this definition,  the assumption that all subspaces assigned to the vertices $v \in V(G)$ have the same rank is dropped, and is replaced by the condition that the average dimension of those subspaces is $r$.

\begin{lemma}[{\cite[Corollary of Lemma 6.13.1]{Roberson2013VariationsOA}}] Suppose $G$ is a vertex-transitive graph. Then $\overline{\xi_f}(G)$ equals to the infimum over all fractions $d/r$ with $d,r \in \NN $ for which there exists an assignment of subspaces $W_v \leq \mathbb{C}^{d}$  to each vertex $v$ of $G$ such that non-adjacent vertices are assigned orthogonal subspaces and $$\frac{1}{|V(G)|}\sum_{v \in V(G)} \dim(W_v) =r. $$
\end{lemma}

\begin{corollary}\label{sumDimLemma}
    Suppose $G$ is a vertex-transitive graph, and $\{W_v\}_{v \in V(G)}$ is an assignment of subspaces of $\mathbb{C}^d$ such that non-adjacent vertices receive orthogonal subspaces. Then \[ \sum_{v \in V(G)}\dim W_v \leq \frac{d \cdot |V(G)|}{\overline{\xi_f}(G)}.\]
\end{corollary}

 From $E_{n/1} = E_n$ and $\overline{\xi_f} \in \Delta(\mathcal{G})$, we directly obtain the following base case:

\begin{lemma}\label{BaseCase}
    For every $n \in \NN$, $\overline{\xi_f}({E_{n/1}}) =n$.
\end{lemma}

The next \autoref{inductionStep} we will use to carry out the induction step in the proof of \autoref{xilemma}.

\begin{lemma}\label{inductionStep}
Let $p,q \in \NN$ with $2 \leq q \leq \frac{p}{2} -1.$
    If $$\overline{\xi_f}\left(E_{p/(q-1)}\right)= \frac{p}{q-1} \qquad\text{and}\,\qquad \overline{\xi_f}\left(E_{p/(q+1)}\right)= \frac{p}{q+1},$$ then $$\overline{\xi_f}\left(E_{p/q}\right)= \frac{p}{q}.$$
\end{lemma}
\begin{proof}
    Suppose $d,r \in \NN$ are such that there exists an assignment of subspaces $W_a \leq \mathbb{C}^d$ of dimension $r$ to each vertex~$a$ of $E_{p/q}$ such that non-adjacent vertices receive orthogonal subspaces. Our aim is to show that $d/r \geq p/q$. We will denote  the distance modulo~$p$ between $a\in \mathbb{Z}_p$ and $b \in\mathbb{Z}_p$ by $|a-b|_p$.  So $W_a \perp W_b$ whenever $|a-b|_p \geq q$. Then, notice that \[
    (W_a \cap W_{a+1}) \perp (W_b \cap W_{b+1})
    \]
    whenever $|a-b|_p \geq q-1$.  Indeed, if $|a-b|_p \geq q-1$, then one of $|a-(b+1)|_p \geq q$ or $|a +1 -b|_p \geq q$ must hold. Thus, assigning $W_a \cap W_{a+1}$ to $a \in \mathbb{Z}_p$ gives rise to an assignment of subspaces to the vertices of $E_{p/(q-1)}$  such that non-adjacent vertices receive orthogonal subspaces. Hence, by  \autoref{sumDimLemma} 
    $$ \sum_{a \in \mathbb{Z}_p} \dim(W_a \cap W_{a+1}) \leq d \cdot \frac{|V(E_{p/(q-1)}|}{\overline{\xi_f}(E_{p/(q-1)})} = d \cdot \frac{p}{p/(q-1)} = d(q-1),$$ 
    whence $$d \geq \frac{1}{q-1}\sum_{a \in \mathbb{Z}_p} \dim(W_a \cap W_{a+1}).$$

    Notice also that  $(W_a + W_{a+1}) \perp (W_b + W_{b+1}) $ whenever $|a-b|_p \geq q+1$. Indeed, if $|a-b|_p \geq q+1$, then $|a-(b+1)|_p \geq q $, $|(a+1)-(b+1)|_p \geq q$ and  $ |(a+1)-b|_p\geq q$. Hence, any vector in $W_a \cup W_{a+1}$ is orthogonal to any vector in $W_b \cup W_{b+1}$, which implies the claim.  Thereby, if we assign $W_{a} + W_{a+1}$  to $a \in \ZZ_p$, then non-adjacent vertices in $ E_{p/(q+1)}$ are assigned orthogonal subspaces. Thus, by a similar reasoning, we have $$d \geq \frac{1}{q+1} \sum_{a \in \mathbb{Z}_p} \dim(W_a + W_{a+1}).$$
    Now, notice that for every $a\in \ZZ_p$, $$\dim(W_a + W_{a+1}) + \dim(W_a \cap W_{a+1}) = \dim(W_a) + \dim(W_{a+1}) =  2r.$$  Hence 
    \[
    2d = \frac{q+1}{q}d + \frac{q-1}{q}d \geq \frac{1}{q} \Bigl(\sum_{a \in \mathbb{Z}_p} \dim(W_a + W_{a+1}) + \dim(W_a \cap W_{a +1}) \Bigr) = \frac{2pr}{q}.
    \]
    Therefore, $\frac{d}{r}\geq \frac{p}{q}$. This shows that $\overline{\xi_f}(E_{p/q}) \geq \frac{p}{q}.$
\end{proof}
\begin{corollary}
    For all integers $n \geq 1$ and $1 \leq q \leq 2^{n-1}$, we have $$\overline{\xi_{f}}(E_{2^n/q}) = \frac{2^{n}}{q}.$$
   
\end{corollary}
\begin{proof}
    Proceed by induction on $n$.
    For $n=1$, we have $\overline{\xi_f}(E_{2/1}) = 2$  by \autoref{BaseCase}. Assume that the claim holds for some $n$. %
    Consider $ 1 \leq q \leq 2^{n}$. If $q$ is even, then $\overline{\xi_f}(E_{2^{n+1}/q}) =\overline{\xi_f}(E_\frac{2^n}{q/2}) = \frac{2^n}{q/2}$ by the induction hypothesis. If $q =1$, then the property holds by \autoref{BaseCase}.
    
    So assume $q$ is odd and $q \geq 3$. Note that $q-1$, $q+1$ are even, with $q +1 \leq 2^{n}$, $q-1 \geq 1$ and therefore, by induction hypothesis, we have $$\overline{\xi_f}\bigl(E_\frac{2^{n+1}}{q-1}\bigr) = \overline{\xi_f}\bigl(E_\frac{2^n}{(q-1)/2}\bigr) = \frac{2^{n+1}}{q-1}$$ and 
   \[   \overline{\xi_f}\bigl(E_{\frac{2^{n+1}}{q+1}}\bigr) = \overline{\xi_f}\bigl(E_{\frac{2^{n}}{(q+1)/2}}\bigr) =
   \frac{2^{n+1}}{q+1}.\] 
   Hence by \autoref{inductionStep}, we also have  \[\overline{\xi_f}\bigl(E_{\frac{2^{n+1}}{q}} \bigr)=\frac{2^{n+1}}{q}.\qedhere
   \]

\end{proof}

\begin{proof}[Proof of \autoref{xilemma}]
    Note that for any rational number $p/q \geq 1$ and $\varepsilon >0$ there exists an $n \geq 1$ and $1 \leq q' \leq 2^n$ such that $2^n /q'$ approximates $ p/q$ within $\varepsilon$.

    Let $p/q$ be any rational number and let $p_n/q_n$ be a sequence of rational numbers with $p_n =2^m$ and $1 \leq q_n \leq 2^{m-1}  $ for some $m$,  that converges to $p/q$ from below. Notice that $E_{p/q} \geq_{\mathcal{G}} E_{p_n/q_n}$ for all $n$. Then, since $\overline{\xi_f}$ is a spectral point and hence in particular monotone with respect to $\geq_{\mathcal{G}}$, we have
    \[ 
    \overline{\xi_f}(E_{p/q}) \geq \sup_{n} \overline{\xi_f}(E_{p_n/q_n}) = \sup_n \frac{p_n}{q_n}= \frac{p}{q}.\qedhere
    \]
    \end{proof}

\section{A continuous family of fraction-normalised spectral points} \label{sec:fam-spec}
As hinted to in the previous section, \autoref{chilemma} and \autoref{xilemma} enable us to make the following claim.
\begin{lemma}\label{lem:sepG}
    There exist $f_0, f_1 \in \Delta(\mathcal{G})$ such that $f_0(E_{p/q}) = f_1(E_{p/q})= p/q$ for all   $p,q\in \mathbb{N}$ with $p/q \geq 2$ and such that there exists a vertex-transitive graph $G$ with $f_0(G) \neq f_1(G)$.
\end{lemma}
\begin{proof}
    A concrete example of such $f_0, f_1$ are the complement of the projective rank $\overline{\xi_f}$ and the fractional clique covering number $\overline{\chi_f}$.  Indeed, these are distinct spectral points  \cite{mancinska2014graph} which take the value $p/q$ on $E_{p/q}$ for all $p/q \in \mathbb{Q}_{\geq 2}$ by \autoref{chilemma} and \autoref{xilemma}. By \autoref{lem:vertexTrans-Example}, there exists a vertex-transitive graph such that $f_0(G) \neq f_1(G).$
    \end{proof}
 \begin{remark} It is also possible to give a concrete example of a graph $G$ that satisfies the conditions of \autoref{lem:sepG}.
    We can take the graph $G$ to be the complement of an orthogonality graph $\Omega_n$, which has vertex set $\{-1,1\}^n$ and $v \sim w$ if and only if $v^Tw = 0$. Indeed, it can be deduced from \cite{wocjan2018spectral} and \cite{mancinska2014graph} that for suitable $n$ we then have $\mathcal{\xi}_f (\overline{\Omega_n}) < \overline{\chi_f}(\overline{\Omega_n})$. Assume that $n$ is a multiple of four, so that $\Omega_n$ is not empty or bipartite. It is  known that $\xi_f(\Omega_n) = n$ \cite{wocjan2018spectral}. On the other hand, we also have $\alpha(\Omega_n) \leq \theta(\Omega_n) = \frac{2^n}{n}$ \cite{mancinska2014graph}, so that $\chi_f(\Omega_n) \geq \frac{|V(\Omega_n)|}{\alpha(\Omega_n)} \geq 2^n / \lfloor{\frac{2^n}{n}}\rfloor$. This is strictly larger than $n$ whenever $n$ does not divide $2^n$, so we immediately find that $\chi_f(\Omega_n) > \xi_f(\Omega_n)$ for $n$ any multiple of four that is at least 12 and not a power of two. 
    \end{remark}

With the work done in \autoref{sec:asd}, \autoref{lem:sepG}  implies the following. 
\begin{lemma} \label{lem:sepG-sequel}
There exists a graph $G$ and $s<t$ such that for any $r \in [s,t]$ there  exists an $f \in \Delta(\mathcal{G})$  that satisfies the two properties below:
\begin{enumerate}[\upshape(1)]
    \item $f(E_{p/q}) = p/q $ for all $p/q \in \QQ_{\geq 2}.$
    \item $f(G) = r$.
\end{enumerate}
\end{lemma}

\begin{proof}

Consider the vertex-transitive graph $G$ and the spectral points $f_0,f_1$ as given in \autoref{lem:sepG}. Let $s=f_0(G) $, $t=f_1(G) $ and let $r \in [s,t]$. 
Recall that by \autoref{interpol}, for any $\lambda \in [0,1]$ one can find a spectral point $f_\lambda$ such that for every vertex-transitive graph~$H$,
$$f_{\lambda}(H) = f_0(H)^{1-\lambda} f_1(H)^{\lambda}.$$
Pick $\lambda = (\log r - \log s)  /(\log t - \log s)$. Then, since $G$ and $E_{p/q}$ for $p/q \in \QQ_{\geq 2}$ are vertex-transitive, \[f_\lambda(G) = s^{1-\lambda} t^{\lambda} = r \,\,\,\,\, \text{ and }\,\,\,\,\, f_{\lambda}(E_{p/q}) = (p/q)^{1-\lambda} (p/q)^{\lambda} = p/q. \qedhere \]

\end{proof}

\section{Warm-up: countable antichains and finite preorders} \label{sec:sbr}

Towards the proof of the main theorem, as a warm-up and illustration of how far we get with less involved constructions, we first prove that some simpler suborders are contained in the asymptotic cohomomorphism preorder, namely countable antichains and finite orders. %
\subsubsection*{Countable antichains}
Recall that in any preorder, an antichain is a set of elements in which no two elements are comparable. We show $(\mathcal{G},\asympleq)$ is universal for countable antichains: %
\begin{proposition}\label{CountAnti}
    Any countable antichain can be order-embedded into $(\mathcal{G},\asympleq)$. 
\end{proposition}
\begin{proof}
Let $f_0,f_1 \in \Delta(\mathcal{G})$ be distinct and $G$ a vertex-transitive graph as in \autoref{lem:sepG}. Suppose without loss of generality that $f_0(G) < \frac{r}{s} < f_1(G)$ for some $\frac{r}{s} \in \mathbb{Q}$. Note that the lines $\ell_n(x) = \frac{2rn+sn-s}{rn}x+\frac{2n+1}{n}$ for $n \in \mathbb{Z}_{\geq 1} $ all pass through the point $(\frac{r}{s},\frac{2r}{s}+3)$, and that both coefficients of each line are in $\mathbb{Q}_{\geq 2}$. Furthermore, the slope of $\ell_n$ increases as $n$ increases.
Consider the set of graphs
\[
\mathcal{A} = \{H_n:=E_{(2rn+sn-s)/rn} \boxtimes G \sqcup E_{(2n+1)/n} \mid n >0 \}.
\]
We will show that $\mathcal{A}$ is an antichain, proving the claim.
We see immediately that $f_0(H_n) = \ell_n(f_0(G))$ and $f_1(H_n) = \ell_n(f_1(G))$, using the properties of the asymptotic spectrum. Further note that if $a>b$, then $\ell_a(f_0(G))<\ell_b(f_0(G))$. Indeed, $\ell_a$ has a larger slope than $\ell_b$, and must therefore be below $\ell_b$ to the left of their intersection point. We similarly conclude that $\ell_a(f_1(G))>\ell_b(f_1(G))$. This translates to
\begin{align*}
    f_0(H_a)&<f_0(H_b),\\
    f_1(H_a)&>f_1(H_b).
\end{align*}
Suppose we had $H_a \asympleq H_b$, then $f_1(H_a) \leq f_1(H_b)$, which is not the case. Similarly, $H_b \asympleq H_a$ cannot hold. Thus $H_a$ and $H_b$ are incomparable. Because $a$ and $b$ were arbitrary, no two elements of $\mathcal{A}$ are comparable. %
\end{proof}
\subsubsection*{Finite preorders}
We show that $(\mathcal{G},\asympleq)$ is universal for finite preorders:
\begin{proposition}\label{FinUniv}
    Every finite order $(\mathcal{M},\leq_\mathcal{M})$ can be order-embedded into $(\mathcal{G},\asympleq)$.
\end{proposition}

In light of the following remark, we can not simply copy the argument proving finite universality for the cohomomorphism order (as in~\cite[Chapter~3]{Hell2004GraphsAH}) to prove the statement for the asymptotic order. We can, however, use a generalization of our construction or finite antichains, while still using a few of the ideas from this non-asymptotic argument.

\begin{remark}\label{rem:comp}
We draw a comparison with the proof of the analogous statement for the (non-asymptotic) cohomomorphism order on graphs, which we will recall here for completeness. It is well known that any finite partial preorder  $(\mathcal{M}, \leq_{\mathcal{M}})$ is isomorphic to a suborder of the inclusion order $(\mathcal{P}(X),\subseteq)$ on the power set of some finite set $X \subset \mathbb{N}_{>0}$ ; indeed $x \mapsto \{y \in X\mid y \leq_X x\}$ is easily seen to work. In the case of the cohomomorphism order, we can then leverage the existence of countable antichains as follows: Let $G_1,G_2,\ldots,G_n$ be pairwise incomparable graphs, assumed to all have connected complement. To any finite subset $A$ of $X$, assign the graph $G_A=\sum_{a \in A} G_a$. Then $G_A \leq G_B \iff A \subseteq B$ indeed holds, because a cohomomorphism $\sum_{i \in A} G_i \rightarrow \sum_{j \in B} G_j$ exists iff for all $i \in A$ there is some $j \in B$ so that there is a cohomomorphism $G_i \rightarrow G_j$, in which case $j=i$, and hence $A \subseteq B$, is forced. This embeds the finite order into the cohomomorphism order on graphs.

Sadly, the same property of cohomomorphisms between join does not generalize to asymptotic cohomomorphisms: There indeed exist graphs $H_1,H_2,H_3$ with connected complement sthat are pairwise asymptotically incomparable, but with $H_1 + H_2 \lesssim H_1 + H_3$. To construct such graphs, we can use a similar construction as in the proof of \autoref{CountAnti}, but this time ensuring the coefficients of the lines are natural numbers (we can manage this because we have only a finite antichain here). For concreteness, let $G = C_5$. Define
    \[
        H_1 = E_6 \boxtimes G \sqcup E_7,\quad
        H_2 = E_3 \boxtimes G \sqcup E_{14},\quad
        H_3 = E_{21}.
    \]
    These are the graphs coming from the lines $\ell_1:y=6x+7,\ell_2:y=3x+14,\ell_3:y=21$ which intersect in $x = \frac{7}{3}$. Because $\chi(C_5)>\frac{7}{3}>\bar{\theta}(C_5)$, an argument as in \autoref{CountAnti} proves that the graphs $H_1,H_2,H_3$ form an asymptotic antichain.
    
    The main observation is that for any spectral point $\phi$, we have that $\phi(H_i) = \ell_i(\phi(G))$, so that either $\phi(H_1) \geq \phi(H_2) \geq \phi(H_3)$ or $\phi(H_1) \leq \phi(H_2) \leq \phi(H_3)$ has to hold (if $\phi(G) \geq \frac{7}{3}$ or $\phi(G) \leq \frac{7}{3}$, respectively). In the first case, we have $\phi(H_1 + H_2) = \phi(H_1) = \phi(H_1 + H_3)$, by \autoref{SumOfGraphs}. In the second case, $\phi(H_1 + H_2) = \phi(H_2) \leq \phi(H_3) = \phi(H_1+H_3)$. So either way, $\phi(H_1 + H_2) \leq \phi(H_1+H_3)$, so $H_1+H_2 \lesssim H_1+H_3$ by \autoref{DualityThm}.
\end{remark}

\begin{proof}[Proof of \autoref{FinUniv}]
 As explained in \autoref{rem:comp}, we can assume that $\mathcal{M} \subseteq  \mathcal{P}(X)$ and $ \leq_\mathcal{M} =\subseteq$ for some finite set $X$. The goal is therefore to assign graphs $G_M$ to sets $M \in \mathcal{M}$ so that $G_M \asympleq G_N$ if and only if $M \subseteq N$. Let $f_0,f_1,G$ be as in \autoref{lem:sepG}, and assume $f_0(G)<f_1(G)$. We will also need $f_\lambda$ as constructed in \autoref{interpol}. We immediately note that $f_r(G) < f_s(G)$ if $r<s$ for $r,s \in [0,1]$, and that $f_{\lambda}$ also satisfies $f_{\lambda}(E_{p/q})=\frac{p}{q}$ for any $\lambda \in [0,1]$.

All subsets in $\mathcal{M}$ are contained in $\mathcal{P}(X)$ which we can assume to be $\{1,2,\ldots,n\}$ for some $n$ after relabelling. Now choose any $n$ polynomials $p_1,p_2,\ldots,p_n$ in $\mathbb{Q}_{\geq 2}[X]$ such that the values $p_1(x),p_2(x),\ldots,p_n(x)$ take on every possible ordering when we let $x$ range over the $n!$ points $\{f_\frac{i}{n!}(G) \mid 1 \leq i \leq n!\}$. This means that for any $\sigma \in S_n$ there exists $x_i\in  \{f_\frac{i}{n!}(G) \mid 1 \leq i \leq n!\}$ such that $$p_{\sigma(1)}(x_i)>p_{\sigma(2)}(x_i)>\cdots >p_{\sigma(n)}(x_i).$$ 
These can indeed be constructed: by interpolation, there certainly exist $n$ real polynomials satisfying this, which we can uniformly approximate on the relevant interval by rational polynomials (and we can ensure the coefficients are in $\mathbb{Q}_{\geq 2}$ by adding a suitable polynomial to every of these approximating polynomials).

Having constructed such polynomials, assign to $m \in \{1,\dots,n\}$ the graph $G_m$ which is the image of $p_m$ under the following mapping:
\[
\sum_{i=0}^d a_iX^i \mapsto \bigsqcup_{i=0}^d E_{a_i}\boxtimes (G^{\boxtimes i}).
\]
We can then assign to any $M \in \mathcal{M} $ the graph $G_M = \sum_{m \in M} G_m$. We note that $f_{\lambda}(G_m) = p_m(f_{\lambda}(G))$, and that $f_{\lambda}(G_M) = \max_{m \in M} f_{\lambda}(G_m) =\max_{m \in M} p_M(f_{\lambda}(G))$, with the first equality here coming from \autoref{SumOfGraphs}.

We only need to prove that these graphs satisfy the condition \[G_M \asympleq G_N \iff M \subseteq N.\] Necessity is obvious, as there is a (non-asymptotic) cohomomorphism from $G_M$ to $G_N$ if $M \subseteq N$.  We now show that if $M$ and $N$ are incomparable, then so are $G_M$ and $G_N$. Indeed, there is some $a \in M\setminus N$ and some $b \in N\setminus M$. Now, we have a point $f_\frac{i}{n!}(G)$ where $p_a$ is the greatest polynomial and a point $f_\frac{j}{n!}(G)$ where $p_b$ is the greatest. By \autoref{SumOfGraphs} and using properties of the spectrum, we find that $f_{\frac{i}{n!}}(G_M) = \max_{m \in M} f_{\frac{i}{n!}}(G_m) = f_{\frac{i}{n!}}(G_a)$ and $f_{\frac{i}{n!}}(G_N) = \max_{m' \in N} f_\frac{i}{n!}(G_{m'}) < f_\frac{i}{n!}(G_a)$  (because $a \not \in N$ while $G_a$ has the greatest value for $f_\frac{i}{n!}$ by construction). We conclude that $G_M \not \asympleq G_N$, and analogously looking at the point $f_\frac{j}{n!}(G)$ we conclude that $G_N \not \asympleq G_M$, finishing the proof.
\end{proof}

\section{Proof of \autoref{th:main}}   \label{sec:ups}

 To prove that every finite preorder can be embedded into $ (\mathcal{G}, \asympleq)$ (see \autoref{FinUniv}), we used the fact that any finite preorder can be order-embedded into a preorder of the form $(\mathcal{P}(X), \subseteq)$ for some finite set $X$. Therefore, it was enough to  order-embed preorders of the form $(\mathcal{P}(X), \subseteq)$ into $ (\mathcal{G}, \asympleq)$. In the proof of \autoref{th:main}, instead of $(\mathcal{P}(X), \subseteq)$ we will employ a more complicated preorder defined on the class of all finite antichains  in the power set of binary words $W=\{0,1\}^*$,  which is known to be  countably universal. As in \autoref{FinUniv}, to every set $A$ we will assign a graph of the form $G_A = \sum_{a 
\in A} G_a$.  Instead of polynomials, we will use a set of lines encoding the binary words.    
The construction of these lines will proceed by induction with respect to the length-lexicographic (shortlex) order.

\subsubsection*{A countably universal preorder on sets of binary words}

 Let  $W= \{0, 1\}^*$ be the set of all finite words over the alphabet $\{0, 1\}$. For words $w, w'$ we write $w \leq_W w'$ if and only if $w'$ is a prefix (i.e.~initial segment)  of~$w$. 
For example, ${011000} \leq_W {011}$ and $010111 \nleq_W 011.$

Let $\mathcal{W}$ be the class of all finite subsets $ A$ of $W=\{0, 1\}^*$ such that no distinct words $w, w'$ in $A$ satisfy $w \leq_W w'$ (i.e.~$A$ is an antichain). For $A,B \in \mathcal{W}$, we write $A  \leq_\mathcal{W} B$ if and only if for each $a \in A$ there exists $b \in B$ such that $a \leq_W b.$

\begin{theorem}[\cite{HubickaNestrilUniversal}, Corollary 2.6]
    $(\mathcal{W}, \leq_{\mathcal{W}})$ is countably universal.
\end{theorem}

 By embedding $(\mathcal{W}, \leq_{\mathcal{W}})$  into $(\mathcal{G}, \asympleq)$, we will show that the latter is universal.

\subsubsection*{Embedding  binary words into graphs}

     Before embedding the entirety of the preorder $(\mathcal{W},\leq_{\mathcal{W}})$ into $(\mathcal{G}, \asympleq)$, we will first assign graphs $G_a$ to each binary string $a$ in the preorder $(W, \leq_W)$. This assignment will respect the order relations in $W$, but also an additional property given by \autoref{existenceOfNiceGa}. 
    \begin{lemma}\label{existenceOfNiceGa}
        We can assign to each binary string $a \in W$ a graph $G_a$  and a ``witness spectral point'' $\phi_a$ such that the following properties hold:
        \begin{enumerate}[\upshape(1)]
        \item For every $v,w\in W$, if $v \leq_W w$, then  $G_v \leq_\mathcal{G} G_w$. 
        \item For every $v, w\in W$, if~$v \nleq_W w$, then  $\phi_v(G_v)>\phi_v(G_w)$, and thus in particular $G_v \not\asympleq G_w$.
        
        \end{enumerate}
    \end{lemma}

    Let us prove \autoref{th:main} assuming \autoref{existenceOfNiceGa} is true.
\begin{proof}[Proof of \autoref{th:main} assuming \autoref{existenceOfNiceGa}]
Assign to each binary string $v \in W$ a graph $G_v$ as in \autoref{existenceOfNiceGa}, and define for all sets $A \in \mathcal{W}$ the graph $G_A = \sum_{a \in A} G_a.$
The sum denotes the graph join of its summands. We claim that $A \mapsto G_A$ is an order-embedding.

Suppose $A \leq_{\mathcal{W}} B$.  We first prove  that $G_A \homomoleq G_B$, which implies $G_A \asympleq G_B$. In $G_A = \sum_{a \in A} G_a$, there are no non-edges between the different components $G_a$ for $a \in A$. Hence, it is enough to cohomomorphically embed each $G_a$ into $G_B$. But since $A \leq_{\mathcal{W}} B$, for each $a \in A$ there is an element $b$ in $B$ such that $a \leq_W b$ and hence $G_a \homomoleq G_b \homomoleq G_B$. Therefore $G_A \homomoleq G_B.$

Next, suppose $A \not\leq_ \mathcal{W}B$. Then, there is a word $a \in A$ such that all words $b$ in $B$ satisfy $a \nleq_W b$.  By \autoref{existenceOfNiceGa}, there exists a witness spectral point $\phi_a$ such that for all $a' \ngeq_W a$, $\phi_a(G_a)> \phi_a(G_{a'})$. Thus, by \autoref{thmVrana}
$$\phi_a(G_A) \geq  \phi_a(G_a)> \max_{b \in B} \phi_a(G_b) = \phi_{a}(G_B).$$ Therefore, we cannot have $G_A \asympleq G_B$. 
\end{proof}

\begin{remark}
    Assuming \autoref{existenceOfNiceGa}, the above proof of \autoref{th:main} also reproves that the preorder $(\mathcal{G}, \homomoleq)$ is universal. Indeed, if $A \leq_{\mathcal{W}} B$, then $G_A \homomoleq G_B$, and if $A \not\leq_{\mathcal{W}} B$ then  %
    $G \not\asympleq H$, so %
    $G \not \homomoleq H.$
\end{remark}
\subsubsection*{Encoding binary words via lines}

 To construct graphs $G_a$ satisfying the properties of \autoref{existenceOfNiceGa}, we will use the asymptotic spectrum duality to ``simulate''  a certain family of lines. As in \autoref{sec:sbr}, we will use them to create incomparability relations between graphs. These lines, equipped with the pointwise comparison preorder, will themselves contain the preorder $(W, \leq_W)$ as an induced suborder and moreover satisfy some additional property.

\begin{lemma} \label{aux}
   Let $1\leq s< t$. We can assign to each word $w \in W$ a pair of rational numbers $(a_w,b_w) \in \mathbb{Q}_{> 2}^2$, a corresponding linear map $\ell_w: x \to a_wx + b_w$ and a ``witness value'' $r_w \in (s,t)$ such that the following conditions are satisfied:
   \begin{enumerate}[\upshape(i)]
   \item For every $v,w \in W$, if $v \leq_W w$ and $v \neq w$, then $a_v <a_w$ and $b_v <b_w$, and thus in particular $\ell_v < \ell_w$ on $(s,t)$. 
   \item For every $v,w \in W$, if $v \nleq_W w$, then $\ell_{v}(r_v) > \ell_w(r_v)$, and in particular $\ell_v \nleq \ell_w$ on $(s,t)$. 
   \end{enumerate}
\end{lemma}

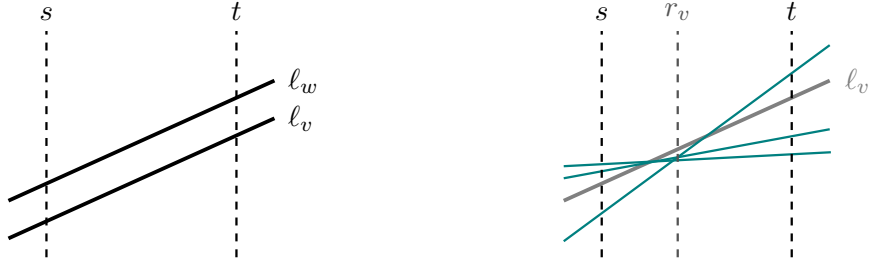
\begin{figure}[h!]
\begin{minipage}{0.5\textwidth}\centering
\begin{tikzpicture}[scale=0.25]

    \coordinate (E) at (7,3);
    \coordinate (F) at (21,9.37);
    \coordinate (A) at (7,1);
    \coordinate (B) at (21,7.37);

    \draw[line width =0.5mm] (E) --(F) node[right, right = 0.3mm] {\(  \ell_w \)}; %
     \draw[line width =0.5mm] (A) --(B) node[right, right = 0.3mm] {\(  \ell_v \)};

    \draw[dashed, ,  line width = 0.3mm] (9,0) -- (9,12) node[above] {\( s \)};
    \draw[dashed,  line width = 0.3mm] (19,0) -- (19,12) node[above] {\( t \)};
\end{tikzpicture}

\end{minipage}
\begin{minipage}{0.5\textwidth}\centering
\begin{tikzpicture}[scale=0.25]
    \coordinate (A) at (7,0);
    \coordinate (B) at (7,12);
    \coordinate (C) at (21,0);
    \coordinate (D) at (21,12);
    \coordinate (E) at (7,3);
    \coordinate (F) at (21,9.37);
    \coordinate (G) at (7,4.83);
    \coordinate (H) at (21.06,5.57);
    \coordinate (I) at (7,4.19);
    \coordinate (J) at (21,6.79);
    \coordinate (K) at (7,0.85);
    \coordinate (L) at (21,11.25);
    \coordinate (M) at (7.08,1.43);
    \coordinate (N) at (13,0);
    \coordinate (O) at (13,12);

    \draw[gray, line width =0.5mm] (E) --(F) node[right, right = 0.5mm] {\(  \ell_v \)}; %

    \draw[teal, line width = 0.3mm] (G) -- (H); %
    \draw[teal, line width = 0.3mm] (I) -- (J); %
    \draw[teal, line width = 0.3mm] (K) -- (L); %
    \draw[dashed, line width = 0.3mm] (9,0) -- (9,12) node[above] {\( s \)};
    \draw[dashed, line width = 0.3mm] (19,0) -- (19,12) node[above] {\( t \)};
    \draw[dashed, Azure4, line width = 0.3mm] (13,0) -- (13,12) node[above] {\( r_v \)};
\end{tikzpicture}
\end{minipage}
\caption{Left: If $v <_W w$, then $\ell_v$ lies below $\ell_{w}$. Right: The blue lines $\ell_{w}$ for $v \nleq_W w$ have the property that at the witness value $r_v$, the line $\ell_v$ lies above all the blue lines.}\label{fig:lines}
\end{figure}

Now, let us prove \autoref{existenceOfNiceGa} assuming \autoref{aux}.
The idea is as follows.  For every $w \in W$, the line $\ell_w(x)  = a_w x + b_w$ will give rise to a graph $G_w:= E_{a_w} \boxtimes G  \sqcup E_{b_w}$ for some fixed vertex-transitive graph $G$ as in \autoref{lem:sepG-sequel}.  The point $r_w$ will be related to the witness spectral point $\phi_w$  of $G_w$.  Using \autoref{lem:sepG-sequel}, the latter spectral point will be constructed so that for all words $w,v \in W$, $\phi_w(G) = r_w$, $\phi_w(E_{p/q})=p/q$ for all $p/q \geq 2 $ and thus $\phi_w(G_v)  = a_vr_w + b_v = \ell_v(r_w)$.

\begin{proof}[Proof of \autoref{existenceOfNiceGa} assuming \autoref{aux}. ]

Let $G$, $s,t$ with $s <t$ be as in \autoref{lem:sepG-sequel}. Suppose that for all $ w\in W$ the rational numbers $a_w,b_w$, linear maps $\ell_w$ and witness values $r_w$ are as in \autoref{aux}. For every word $w \in W$, set $$G_w:= E_{a_w} \boxtimes G  \sqcup E_{b_w}.$$ 

If $w, v$ are words with $w \leq_W v$, then $a_w \leq a_v$ and $b_w \leq b_v$, whence also $E_{a_w} \homomoleq E_{a_v}$ and $E_{b_w} \homomoleq E_{b_v}$, and thereby by \autoref{aux} (see the left part of \autoref{fig:lines}), $$G_{w} \homomoleq G_{v}.$$ This proves claim (i).

  Let us prove claim (ii). By \autoref{lem:sepG-sequel} and our choice of $G$, for all $v$ there is a spectral point $\phi_v$ such that for all $p/q \in \mathbb{Q}_{\geq 2}$, $\phi_v(E_{p/q}) = p/q $ and $\phi_{v}(G) = r_v $. Thus for any $v,w \in W$, by  \autoref{aux} (see the right part of \autoref{fig:lines}), $$\phi_v(G_w)  =\phi_v(E_{a_w})\phi_v(G) +\phi_w(b_w)= a_w r_v + b_w = \ell_w(r_v).$$ Hence, if $ v \nleq_W w$, then \[\phi_v(G_v) = \ell_v(r_v) >  \ell_w(r_v) = \phi_v(G_w). \qedhere\] 
\end{proof}

Before presenting a rigorous proof of \autoref{aux}, let us first give an informal description of the strategy. We will process the strings from shortest to longest, starting with $\emptyset$, then $0$ and $1$, moving to $00$ and $01$, $10$ and $11$ and so on. For convenience, define  the  \emph{shortlex} order $\mathcal{L}$ on $W$: for $w,v \in W$, let $w \leq_{\mathcal{L}} v $ if either $w$ is strictly shorter than~$v$, or they are of equal length and $w$ is below $v$ in the lexicographical order. 

At each step,  we will construct the sets  $\{a_w,b_w,r_w,\ell_w\}$  (consisting of  rational numbers $a_w,b_w,r_w$ and lines $\ell_w$) in pairs.  More precisely,  once we reach the word~$u0$, we will choose the parameters corresponding to the words $u0, u1$ at the same time. In addition, we will choose $a_{u0},a_{u1}, b_{u0}, b_{u1}, r_{u0},r_{u1}$ to lie  very close to $a_u, b_u,r_u$ respectively.
\vskip2ex

\noindent {\bf Example (Construction of the first members).}  The line $\ell_\emptyset$ can be any line of the form $\ell_\emptyset(x) = a_\emptyset x + b_\emptyset$ such that $a_\emptyset, b_\emptyset >2$, and $r_\emptyset$ is any value in $(s,t)$ (see \autoref{ConstrLES}).

For the next words $0$ and $1$, since $0,1 <_W\emptyset$, we need to ensure that $2<  a_0,a_1 < a_\emptyset$ and $2 < b_0,b_1 < b_\emptyset$. The lines $\ell_0,\ell_1$ need to be  situated strictly below $\ell_\emptyset$ on $[s,t]$. Moreover, since $0\nleq _W 1$, we need a point $r_0 
\in (s,t)$ such that $\ell_0(r_0) > \ell_1(r_0)$. Likewise, we need a point $r_1 \in (s,t)$ such that $\ell_1(r_1) > \ell_0(r_1)$. For convenience, we can make $\ell_1, \ell_0$ intersect at $r_\emptyset$ and choose some $r_0$ to the right of $r_\emptyset$ and some $r_1$ to the left of $r_\emptyset$. (See \autoref{ConstrL10}.)

Now move on to the construction of the lines $\ell_{00}$,   $ \ell_{01}$. At this point the picture starts looking more convoluted. We require the rational numbers  $a_{00},a_{01},b_{00},b_{01}, \ell_{00}, \ell_{01}$ to respect the order $\emptyset >_W 0 >_W 00, 01$. To achieve this, it is enough to ensure that $a_{00},a_{01}< a_0$ and $b_{00},b_{01}< b_0$. Then, the line $\ell_{00}$ will  lie strictly below $ \ell_{0}$. In addition, we need to make sure that $r_{1}$ witnesses that $1 \nleq_W 00, 01$ by having $\ell_1(r_1) > \ell_{00}(r_1), \ell_{01}(r_1)$. Moreover, we would like to find points $r_{00}$, $r_{01}$ that would witness $00, 01 \nleq_W 1$. This can be achieved by constructing $\ell_{00}$, $\ell_{01}$  and  $r_{00}$, $r_{01}$  as small  displacements of the lines $\ell_0$ and   $r_0$ respectively. Then the necessary conditions will  follow by continuity. (See \autoref{ConstrL0001}).  The next step, i.e. the construction of lines  $l_{10}$, $l_{11}$, is illustrated in   \autoref{ConstrL1011}.

  \begin{figure} 
  \begin{minipage}{0.95\textwidth}
\centering
 \begin{tikzpicture}[xscale = 0.8,yscale = 0.2]
  \tkzInit[xmin=-18,xmax=0,ymin=-7,ymax=17]
  \tkzClip

   \tkzDefPoints{-16.5/-1/T1, -16.5/16.5/T2}
  \tkzDrawLine[line width = 0.4mm, dashed](T1,T2)
  \tkzLabelPoint[left,dashed,font=\large](T1){$s$}

  \tkzDefPoints{-1/-1/T1, -1/16/T2}
  \tkzDrawLine[line width = 0.4mm, dashed](T1,T2)
  \tkzLabelPoint[right,dashed,font=\large](T1){$t$}

  \tkzDefPoint(5,1){A}
  \tkzDefPoint(5,18){B}
  \tkzDrawLine[add=5 and 5](A,B)

  \tkzDefPoint(24,1){C}
  \tkzDefPoint(24,18){D}
  \tkzDrawLine[add=5 and 5, thick](C,D)

  \tkzDefPoints{-16.5/14/F, -1/14/G}
  \tkzDrawLine[thick,color=red, line width = 0.5mm ](F,G)

  \tkzLabelPoint[above left,font=\large, red](F){$
  \ell_{\emptyset}\,\,\,$}

  \tkzDefPoints{-9/-1/Q1, -9/16/Q2}
  \tkzDrawLine[thick, dashed, red, line width= 0.45 mm](Q1,Q2)
  \tkzLabelPoint[right,font=\large, red](Q2){$r_{\emptyset}$}

\end{tikzpicture}\caption{At first, we construct the line  $\ell_\emptyset$. Any point $r_{\emptyset} \in (s,t)$ can be chosen as witness point.}\label{ConstrLES}
  \end{minipage}
    \end{figure}
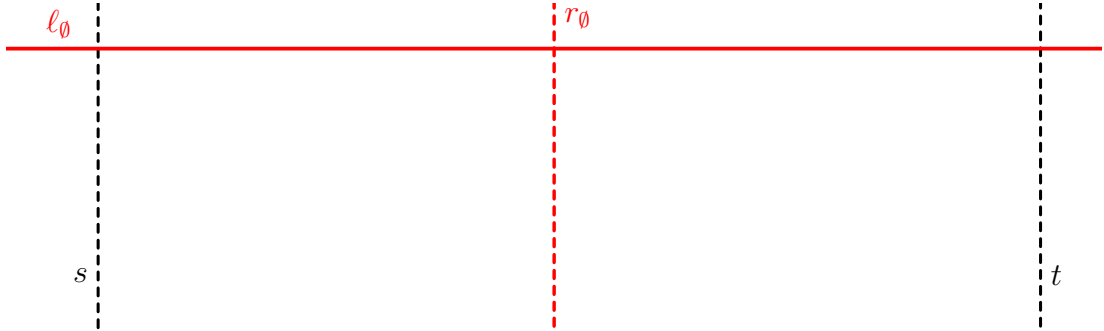

   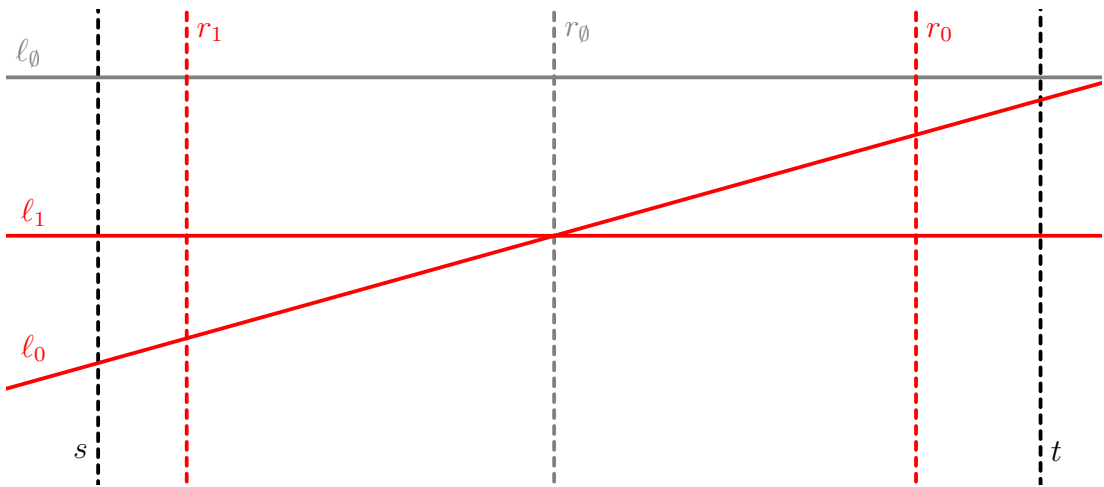
\begin{figure} 
  
\centering

\begin{tikzpicture}[xscale = 0.8,yscale = 0.3]
  \tkzInit[xmin=-18,xmax=0,ymin=-4,ymax=17]
  \tkzClip

     \tkzDefPoints{-16.5/-2.5/T1, -16.5/16.5/T2}
  \tkzDrawLine[line width = 0.5mm, dashed](T1,T2)
  \tkzLabelPoint[left,dashed,font=\large](T1){$s$}

  \tkzDefPoints{-1/-2.5/T1, -1/16/T2}
  \tkzDrawLine[line width = 0.5mm, dashed](T1,T2)
  \tkzLabelPoint[right,dashed,font=\large](T1){$t$}
  \tkzDefPoints{-17/14/F, -1/14/G}
  \tkzDrawLine[thick,color=gray, line width = 0.5mm ](F,G)

  \tkzLabelPoint[above left,font=\large, gray](F){$
  \ell_{\emptyset}\,\,\,$}

  \tkzDefPoints{-17/7/H, -1/7/I}

  \tkzDefPoints{-9/-7/Q1, -9/16/Q2}
  \tkzDrawLine[thick, dashed,gray, line width = 0.5mm](Q1,Q2)
  \tkzLabelPoint[right,font=\large, gray](Q2){$r_{\emptyset}$}

  \tkzDefPoints{-15.043/-7/P1, -15.043/16/P2}
  \tkzDrawLine[thick, dashed, red, line width = 0.5mm](P1,P2)
  \tkzLabelPoint[right,font=\large, red](P2){$r_{1}$}

  \tkzDefPoints{-3.048/-7/R1, -3.048/16/R2}
  \tkzDrawLine[thick, dashed, red, line width = 0.5mm](R1,R2) 
  \tkzLabelPoint[right,font=\large, red](R2){$r_{0}$}

  \tkzDefPoint(5,1){A}
  \tkzDefPoint(5,18){B}
  \tkzDrawLine[add=5 and 5](A,B)

  \tkzDefPoint(24,1){C}
  \tkzDefPoint(24,18){D}
  \tkzDrawLine[add=5 and 5, thick](C,D)

\tkzDrawLine[thick,color=red, line width = 0.5mm](H,I)
  \tkzLabelPoint[above left,font=\large, red](H){$
  \ell_{1}\,\,$}

  \tkzDefPoints{-17/1/J, -1/13/K}
  \tkzDrawLine[thick,color=red, line width = 0.5mm](J,K)
  \tkzLabelPoint[above left,font=\large, red](J){$
  \ell_{0}\,\,$}

\end{tikzpicture}\caption{Lines $
\ell_0, \ell_1$ are situated strictly below $\ell_\emptyset$ on $[s,t]$. For simplicity we let them intersect at $r_\emptyset$. We pick $r_1$ to the left of the intersection point $r_\emptyset$  so that $\ell_1(r_1) > \ell_0(r_1)$. Analogously, we choose $r_0$ to the right of the the intersection $r_{\emptyset}$ point so that $\ell_0(r_0) > \ell_1(r_0)$. }\label{ConstrL10}

    \end{figure}

   \begin{figure} 
   \centering

\begin{tikzpicture}[xscale = 0.8,yscale = 0.3]
  \tkzInit[xmin=-18,xmax=0,ymin=-7,ymax=17]
  \tkzClip

    \tkzDefPoints{-16.5/-5.5/T1, -16.5/16.5/T2}
  \tkzDrawLine[line width = 0.5mm, dashed](T1,T2)
  \tkzLabelPoint[ below left,dashed,font=\large](T1){$s$}

  \tkzDefPoints{-1/-5.5/T1, -1/16/T2}
  \tkzDrawLine[line width = 0.5mm, dashed](T1,T2)
  \tkzLabelPoint[below right,dashed,font=\large](T1){$t$}

  \tkzDefPoint(5,1){A}
  \tkzDefPoint(5,18){B}
  \tkzDrawLine[add=5 and 5](A,B)

  \tkzDefPoint(24,1){C}
  \tkzDefPoint(24,18){D}
  \tkzDrawLine[add=5 and 5, thick](C,D)

  \tkzDefPoints{-15.043/-7/P1, -15.043/16/P2}
  \tkzDrawLine[ dashed, gray, line width = 0.5mm](P1,P2)
  \tkzLabelPoint[right,font=\large, gray](P2){$r_{1}$}

  \tkzDefPoints{-9/-7/Q1, -9/16/Q2}
  \tkzDrawLine[dashed, gray, line width = 0.5mm](Q1,Q2)
  \tkzLabelPoint[right,font=\large, gray](Q2){$r_{\emptyset}$}

  \tkzDefPoints{-3.048/-7/R1, -3.048/16/R2}
  \tkzDrawLine[thick, dashed, gray, line width = 0.5mm](R1,R2) 
  \tkzLabelPoint[right,font=\large, gray](R2){$r_{0}$}

  \tkzDefPoints{-17/14/F, -1/14/G}
  \tkzDrawLine[thick,gray, line width = 0.5mm](F,G)

  \tkzLabelPoint[above left,font=\large, gray](F){$
  \ell_{\emptyset}\,\,\,$}

  \tkzDefPoints{-17/7/H, -1/7/I}
  \tkzDrawLine[thick,color=gray, line width = 0.5mm](H,I)
  \tkzLabelPoint[above left,font=\large, gray](H){$
  \ell_{1}\,\,$}

  \tkzDefPoints{-17/1/J, -1/13/K}
  \tkzDrawLine[thick,color=gray, line width = 0.5mm](J,K)
  \tkzLabelPoint[above left,font=\large, gray](J){$
  \ell_{0}\,\,$}

   \tkzDefPoints{-1/12/P, -17/-0/Q}
  \tkzDrawLine[thick,color=red, line width = 0.5mm](P,Q)
  \tkzDefPoints{-1/12/P, -17/-0.5/Q}
  \tkzLabelPoint[ below left,font=\large, red](Q){$
  \ell_{01}$}

  \tkzDefPoints{-0.99/12.45/R, -17/-3/S}
  \tkzDrawLine[thick,color=red, line width = 0.5mm](R,S)
  \tkzDefPoints{-0.99/12.45/R, -17/-4/S}
  \tkzLabelPoint[below left,font=\large, red](S){$
  \ell_{00}$}

    \tkzDefPoints{-1.92/-7/A1, -1.92/16/A2}
  \tkzDrawLine[thick, dashed,line width = 0.5mm, red](A1,A2)
  \tkzLabelPoint[right,font=\large,red ](A2){$r_{01}$}

  \tkzDefPoints{-4.17/-7/B1, -4.17/16/B2}
  \tkzDrawLine[thick, dashed,line width = 0.5mm, red](B1,B2)
  \tkzLabelPoint[right,font=\large, red](B2){$r_{00}$}

  \tkzDefPoint(24,20){E}
\end{tikzpicture} \caption{We construct the lines $\ell_{01}$ and $\ell_{00}$ as slight ``perturbations'' of the line $\ell_0.$}\label{ConstrL0001}

   \end{figure}
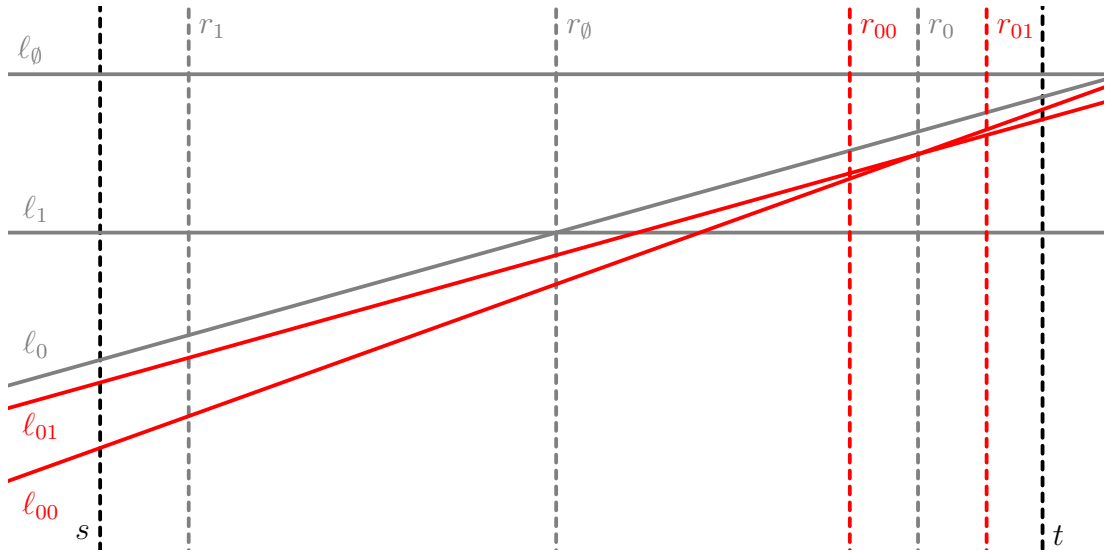

   \begin{figure}

\centering

\begin{tikzpicture}[xscale = 0.8,yscale = 0.4]
  \tkzInit[xmin=-18,xmax=0,ymin=-7,ymax=17]
  \tkzClip

 \tkzDefPoints{-16.5/-5/T1, -16.5/16.5/T2}
  \tkzDrawLine[line width = 0.4mm, dashed](T1,T2)
  \tkzLabelPoint[left,dashed,font=\large](T1){$s$}

  \tkzDefPoints{-1/-5/T1, -1/16/T2}
  \tkzDrawLine[line width = 0.4mm, dashed](T1,T2)
  \tkzLabelPoint[right,dashed,font=\large](T1){$t$}

  \tkzDefPoint(5,1){A}
  \tkzDefPoint(5,18){B}
  \tkzDrawLine[add=5 and 5](A,B)

  \tkzDefPoint(24,1){C}
  \tkzDefPoint(24,18){D}
  \tkzDrawLine[add=5 and 5, thick](C,D)

  \tkzDefPoints{-15.043/-7/P1, -15.043/16/P2}
  \tkzDrawLine[thick, dashed, line width = 0.5mm, gray](P1,P2)
  \tkzLabelPoint[right,font=\large, gray](P2){$r_{1}$}

  \tkzDefPoints{-9/-7/Q1, -9/16/Q2}
  \tkzDrawLine[thick, dashed, gray, line width =0.5mm](Q1,Q2)
  \tkzLabelPoint[right,font=\large, gray](Q2){$r_{\emptyset}$}

  \tkzDefPoints{-3.048/-7/R1, -3.048/16/R2}
  \tkzDrawLine[thick, dashed, gray, line width = 0.5mm](R1,R2) 
  \tkzLabelPoint[right,font=\large, gray, line width=0.5mm](R2){$r_{0}$}

  \tkzDefPoints{-1.92/-7/A1, -1.92/16/A2}
  \tkzDrawLine[thick, dashed, gray, line width = 0.5mm](A1,A2)
  \tkzLabelPoint[right,font=\large, gray](A2){$r_{01}$}

  \tkzDefPoints{-4.17/-7/B1, -4.17/16/B2}
  \tkzDrawLine[thick, dashed, gray, line width = 0.5mm](B1,B2)
  \tkzLabelPoint[right,font=\large, gray](B2){$r_{00}$}

  \tkzDefPoints{-17/14/F, -1/14/G}
  \tkzDrawLine[thick,color=gray, line width = 0.5mm ](F,G)

  \tkzLabelPoint[above left,font=\large, gray](F){$
  \ell_{\emptyset}\,\,\,$}

  \tkzDefPoints{-17/7/H, -1/7/I}
  \tkzDrawLine[thick,color=gray, line width = 0.5mm](H,I)
  \tkzLabelPoint[above left,font=\large, gray](H){$
  \ell_{1}\,\,$}

  \tkzDefPoints{-17/1/J, -1/13/K}
  \tkzDrawLine[thick,color=gray, line width = 0.5mm](J,K)
  \tkzLabelPoint[above left,font=\large, gray](J){$
  \ell_{0}\,\,$}

  \tkzDefPoints{-1/12/P, -17/-1/Q}
  \tkzDrawLine[thick,color=gray, line width = 0.5mm](P,Q)
  \tkzLabelPoint[above left,font=\large, gray](Q){$
  \ell_{01}$}

  \tkzDefPoints{-0.99/12.45/R, -17/-3/S}
  \tkzDrawLine[thick,color=gray, line width = 0.5mm](R,S)
  \tkzLabelPoint[above left,font=\large, gray](S){$
  \ell_{00}$}

  \tkzDefPoints{-17/4/L, 0/4/M}
  \tkzDrawLine[thick,color=red, line width = 0.5mm](L,M)

  \tkzDefPoints{-17/4/L, -1/3.5/M}
  \tkzLabelPoint[below right,font=\large, color = red](M){$\ell_{11}$}

  \tkzDefPoint(-1.029,6.48){N}
  \tkzDefPoint(-16.941,3.66){O}
  \tkzDrawLine[thick,color=red, line width = 0.5mm](N,O)
  \tkzLabelPoint[below right,font=\large, red](N){$
  \ell_{10}$}

    \tkzDefPoints{-14/-7/S1, -14/16/S2}
  \tkzDrawLine[thick, dashed, red, line width = 0.5mm](S1,S2)
  \tkzLabelPoint[right,font=\large,color = red](S2){$r_{10}$}

  \tkzDefPoints{-15.9/-7/T1, -15.9/16/T2}
  \tkzDrawLine[thick, dashed, red, line width = 0.5mm](T1,T2)
  \tkzLabelPoint[right,font=\large, color = red](T2){$r_{11}$}

  \tkzDefPoint(24,20){E}
\end{tikzpicture}

 \caption{We construct lines $\ell_{10}$ and $\ell_{11}$ as slight ``perturbations'' of the line $\ell_1.$}\label{ConstrL1011}

   \end{figure}
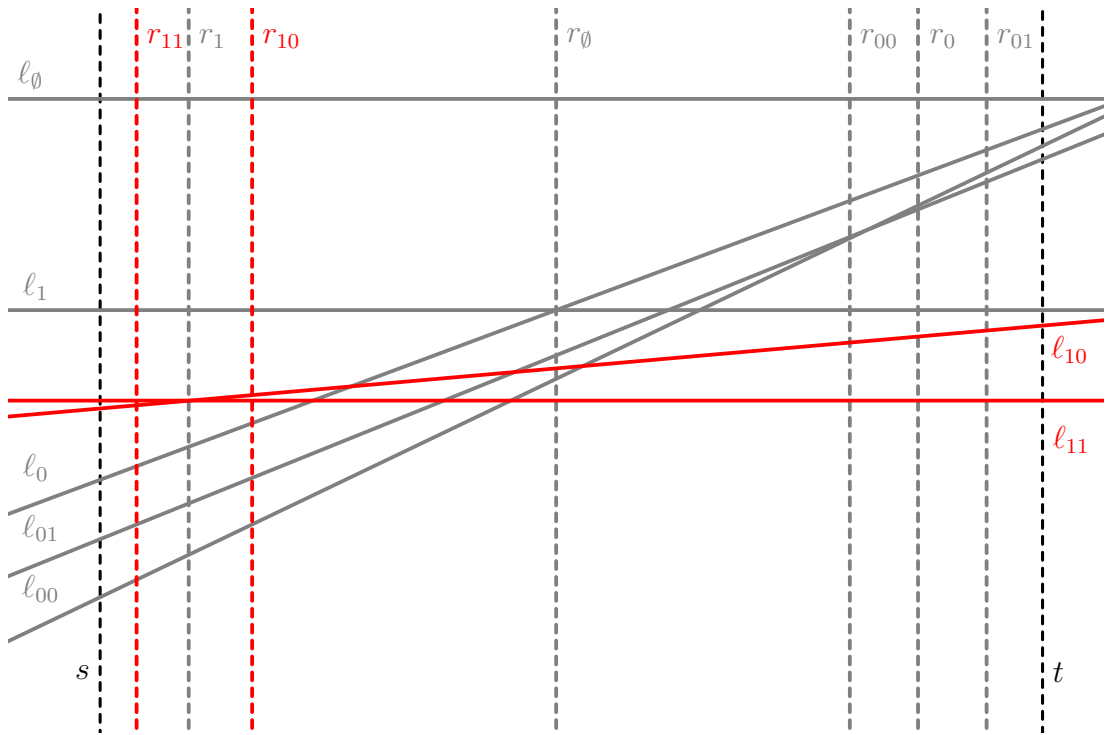

When we follow this approach, the structure of $(W, \leq_W)$ and $(W, \leq_\mathcal{L})$ plays to our advantage. Suppose that a word $w \in W$ is followed directly by the word $u0$ with respect to the order $\mathcal{L}$. The set of words $v \leq_\mathcal{L}w$,  such that $$v \ngeq_W u0 
\,\,\,\,\text{ or }\,\,\,\, v \ngeq_W u1,$$ is precisely the set of words  $v \leq_{\mathcal{L}}w$
 satisfying $$v\ngeq_W u.$$ So, by choosing $a_{u0},a_{u1}, b_{u0}, b_{u1}, r_{u0},r_{u1}$ sufficiently close to their counterparts corresponding to $u$, by continuity the newly introduced parameters will automatically  satisfy most of the required incomparability constraints.

   \vskip2ex
Let us now present a rigorous proof of \autoref{aux}.
   \begin{proof} [Proof of \autoref{aux}.]
   Recall that  $\mathcal{L}$ denotes the (total) shortlex order on $W$.
   As mentioned before, we will construct the parameters $(a_w,b_w,\ell_w)$ recursively and in pairs following the (total) order $
   \mathcal{L}$: first $\emptyset$, then $0,1$, followed by $00$, $01$, and then $10$, $11$, etc.  We will show that the following property $P_u$ holds for all words $u \in W$. %

\begin{framed}
\noindent\textbf{Property} $P_u$: 
We can assign to each word $w \leq_{\mathcal{L}}u$ a pair of rational numbers $(a_w,b_w) \in \mathbb{Q}_{>2}^2$, a corresponding linear map $\ell_w: x \to a_wx + b_w$ and a ``witness value'' $r_w \in (s,t)$ such that the following conditions (which are in correspondence with those of \autoref{aux}) are satisfied:
\begin{enumerate}[\upshape(i)]
   \item For every $v,w \leq_{\mathcal{L}} u$, if $v \leq_W w$ and $v \neq w$, then $a_v <a_w$ and $b_v <b_w$, and thus in particular $\ell_v < \ell_w$ on $(s,t)$. 
   \item For every $v,w \leq_{\mathcal{L}} u$, if $v \nleq_W w$, then $\ell_{v}(r_v) > \ell_w(r_v)$, and in particular $\ell_v \nleq \ell_w$ on $(s,t)$.
\end{enumerate}\vspace{-0.5em}\end{framed}
   
   Start the induction  with the empty string $\emptyset$. We aim to fix rational numbers $a_{\emptyset}, b_{\emptyset}, r_{\emptyset}$ and a line $\ell_\emptyset$. There are no constraints to be satisfied except for the condition $a_\emptyset, b _\emptyset >2$ and $r_\emptyset \in (s,t)$. We can choose $a_{\emptyset}, b_{\emptyset}, r_{\emptyset}$ within those conditions arbitrarily. 
    
Now, suppose we have shown the property for all words up to (with respect to the order $\mathcal{L}$) some $u'$, ending with  $1$. The next word (with respect to the order $\mathcal{L}$) is  of the form $u0$ for some $u\leq_\mathcal{L} {u'}$. By induction hypothesis, there exists an assignment of parameters $(a_w,b_w, \ell_w,r_w)$ to all words $w \leq_{\mathcal{L}} u'$ satisfying properties (i) and (ii) of $P_{u'}$. We will extend this assignment  to all words $w \leq_{\mathcal{L}} u1$ by defining both $a_{u0},b_{u0},r_{u0}$ and $a_{u1},b_{u1},r_{u1}$ so as to keep properties (i) and (ii) satisfied. 
We will choose $a_{u0},b_{u0},r_{u0}$ and $a_{u1},b_{u1},r_{u1}$ to lie very close to $a_u,b_u,r_u$, so that, overall, the lines $\ell_{u0},\ell_{u1}$ behave very similarly to $\ell_u$. 

Let $\mu  \in (0,1)$ be such that $1/\mu = r_u$ (this is possible since $r_u >s \geq 1)$ and for any~$\eps >0$, set \[ a_{u0}(\eps) = a_u - (1-\mu)\eps,\,\,\,\,\,\,\,\,\,\,\,\,\,\,\, b_{u0}(\eps) = b_u - 2\eps, \] \[a_{u1}(\eps) = a_u - \eps, \,\,\,\,\,\,\,\,\,\,\,\,\,\,\,b_{u1}(\eps) =b_u - \eps. \]  
Important features of this construction are the inequalities $a_{u0}(\eps),a_{u1}(\eps) < a_u(\eps)$ and $b_{u0}(\eps),b_{u1}(\eps) < b_u(\eps)$. Furthermore, $a_{u0}(\eps) > a_{u1}(\eps) $ but $b_{u0}(\eps) < b_{u1}(\eps)$.  It is easy to compute that the lines $x \mapsto a_{u0}(\eps) x + b_{u0}(\eps) $ and $x \mapsto a_{u1}(\eps) x + b_{u1}(\eps) $ intersect at $$-\frac{b_{u0}(\eps) -b_{u1}(\eps)}{a_{u0}(\eps) -a_{u1}(\eps) } = \frac{1}\mu = r_u.$$ 
(The reader can observe  the relative position of the lines in   \autoref{positionOfLu0Lu1}.)
\begin{figure}[h!]
 \label{PostionLu1Lu0}
 \begin{center}
\begin{tikzpicture}[xscale=2.5, yscale=1.3 ]
 \clip (0.5,-0) rectangle (4.5,4);
  \draw[dashed,line width =0.5mm] (1,-1) -- (1,3) node[above] {\( s \)};
    \draw[dashed,line width =0.5mm] (3,-1) -- (3,3) node[above] {\( t \)};
    \draw[dashed, Azure4,line width =0.5mm] (1.5,-1) -- (1.5,3) node[above] {\( 1/\mu \)};
    
    \draw[ gray, line width =0.5mm] (-0.5,1.5) -- (3.5,2.5) node[above] {$\ell_u$};
    \draw[ red, line width =0.5mm] (-0.5,0.5) -- (3.5,2) node[above] {$\ell_{u0}$};
    \draw[red, line width =0.5mm] (-0.5,1) -- (3.5,1.5) node[above] {$\ell_{u1}$};
    \draw[<->, violet] (4,1) -- (4,3) node[midway,right=5pt] { small};

\end{tikzpicture}
\caption{The lines $\ell_{u0}$, $\ell_{u1}$ are situated strictly below, but close to~$\ell_u$. 
}
\label{positionOfLu0Lu1}

\end{center}
\end{figure}
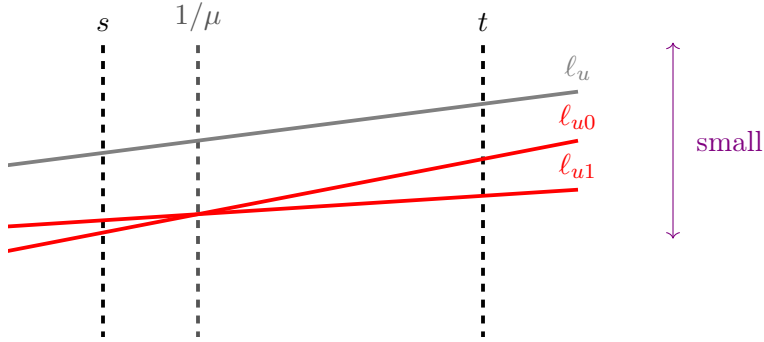

We claim that by setting $a_{u0}=a_{u0}(\eps)$, $a_{u1}=a_{u1}(\eps)$, $b_{u0}=b_{u0}(\eps)$, $b_{u1}=b_{u1}(\eps)$ for $\eps >0 $ small enough and choosing appropriately $r_{u0}$, $r_{u1}$, the parameters will satisfy the following conditions:
    \begin{enumerate}%
    \item[(a)] $2<a_{u0},a_{u1} <a_{u}$ and $2<b_{u0}, b_{u1}< b_u$. This is enough to ensure that property~(i) of $P_{u1}$ holds.

\item[(b1)] For any $v \leq_{\mathcal{L}}  u'$, if $v \nleq_W u$, line $\ell_v$ lies above the newly introduced lines $\ell_{u0}$, $\ell_{u1}$ at the witness value $r_v$. That is, $\ell_v(r_v) > \ell_{u0}(r_v), \ell_{u1}(r_v)$.
    
    \item[(b2)] For any $v \leq_{\mathcal{L}}  u'$, if $v \ngeq_W u$, then  the newly introduced lines $\ell_{u0},\ell_{u1}$ lie above the line $\ell_v$ at the new witness values $r_{u0}$, $r_{u1}$. That is, $\ell_{u0}(r_{u0})> \ell_v(r_{u0})$ and $\ell_{u1}(r_{u1})> \ell_v(r_{u1})$. 
    \item[(b3)] Moreover, $\ell_{u0}(r_{u0})> \ell_{u1}(r_{u0})$ and $\ell_{u1}(r_{u1})> \ell_{u0}(r_{u1})$.
     \end{enumerate}
     \noindent If (b1), (b2) and (b3) hold, then the  parameter assignment satisfies  property~(ii) of~$P_{u1}$.

\noindent As noted above, it is already true by construction that $a_{u0},a_{u1}<a_u$ and $b_{u0},b_{u1}<b_u$.

\noindent For condition (b1), note that for any $v \leq_{\mathcal{L}} u'$, if $v\nleq_W u$, then 
\[
a_{u0}(\eps) r_{v} + b_{u0}(\eps)\xrightarrow[\eps \to 0]{} \ell_u(r_v) < \ell_{v}(r_v).
\]
\noindent For condition (b2), observe that for $v \leq_{\mathcal{L}} u'$, if $v\ngeq_W u$, then
\[ a_{u0}(\eps) r_{u} + b_{u0}(\eps) \xrightarrow[\eps \to 0]{} \ell_u(r_u) >\ell_{v}(r_u).\] 
\noindent The last inequality is valid by  hypothesis. 
Analogous properties hold for $a_{u1}(\eps)$, $ b_{u1}(\eps)$.  We claim that there is $\eps>0$ such that for $v \leq_{\mathcal{L}} u'$  and $v \nleq_W u$, \[\ell_{u0}(r_v), \ell_{u1}(r_v) < \ell_{v}(r_v), \] and whenever $v \ngeq_W u$  \[ \ell_{u0}(r_{u}), \ell_{u1}(r_u)>\ell_{v}(r_u),\]
while $a_{u0},a_{u1},b_{u0},b_{u1}>2$. Indeed, this is a finite set of constraints, particularly because $\{v \in W \mid v \leq_{\mathcal{L}} u'\}$ is finite, and we have observed that all of these constraints are true in the limit $\eps \rightarrow 0$, hence they are simultaneously satisfied for small enough~$\eps>0$.

\begin{figure}[h] 
\begin{tikzpicture}[line cap=round,x=1.5cm,y=1.1cm, xscale=0.6, yscale = 0.8]

\clip(-9,-8)rectangle (6.5,3.5);
\draw [line width=0.5mm,dashed] (-8,-17)-- (-8,3);
\draw [line width=0.5mm] (-8,1)-- (-8,1) node[right,font=\large ] {\( s\)};

\draw [line width=0.5mm, dashed] (5,-17)-- (5,3.5) (5,-6)node[left,font=\large ] {\( t \)};

\draw [line width=0.5mm, Azure2] (-10.10,-5.92)-- (9.25,-4.97);

\draw [line width=0.5mm, Azure2] (-10.09,-6.75)-- (9.24,-3.39);

\draw [line width=0.5mm, Azure2] (-10.04,-14.39)-- (9.21,7.13);

\draw [line width=0.5mm, Azure2] (-10.03,-15.94)-- (9.21,8.46);

\draw [line width=0.5mm,dashed, gray] (-2.9,-17)-- (-2.9,3);
\draw [line width=0.5mm,dashed, gray] (-2.9,1) -- (-2.9,1) node[right,font=\large ] {\( r_{u0} \)};

\draw [line width=0.5mm,dashed, gray] (-1.6,-17)-- (-1.6,3);
\draw [line width=0.5mm,dashed, gray] (-1.6,1)-- (-1.6,1) node[right,font=\large ] {\( r_{u} \)};

\draw [line width=0.5mm,dashed, gray] (-0.3,-17)-- (-0.3,3);
\draw [line width=0.5mm,dashed, gray] (-0.3,1)-- (-0.3,1)
node[right, gray,font=\large ] {\( r_{u1}\)};

\draw [line width=0.5mm, red] (-10.09,-7.77)-- (9.22,6.02);

\draw [line width=0.5mm, red] 
(-5,-3.77)-- (-5,-3.77) node[left,font=\large] {\( \ell_u \)};

\draw [line width=0.5mm ,red] (-10.08,-9.78)-- (9.22,4.13);
\draw [line width=0.5mm ,red] (6,0)-- (6,0)node[left,font=\large] {\( \ell_{u0}\)};

\draw [line width = 0.5mm, red] (-10.08,-8.82)-- (9.22,2.90);
\draw [line width = 0.5mm, red] (6,2.2)-- (6,2.2) node[left,font=\large] {\( \ell_{u1}\)};

\end{tikzpicture}
\caption{The lines $\ell_{u0}$, 
 $\ell_{u1}$ are situated just beneath $\ell_u$. At $r_u$ they are above the lines $\ell_{u'}$ for $w' \ngeq_W w$ (in gray). We pick $r_{u0}$ and $r_{u1}$ in the vicinity of $r_u$, so that $\ell_{u0},\ell_{u1}$ are above the gray lines at $r_{u0}$ and $r_{u1}$ respectively. We pick $r_{u0}$ below the intersection point of $\ell_{u0}$ and $\ell_{u1}$ because we want $\ell_{u0}$  to lie above $\ell_{u1}$ at $r_{u0}$. Meanwhile, we pick $r_{u1}$ on the right of this intersection.}\label{summaryPic}
\end{figure}
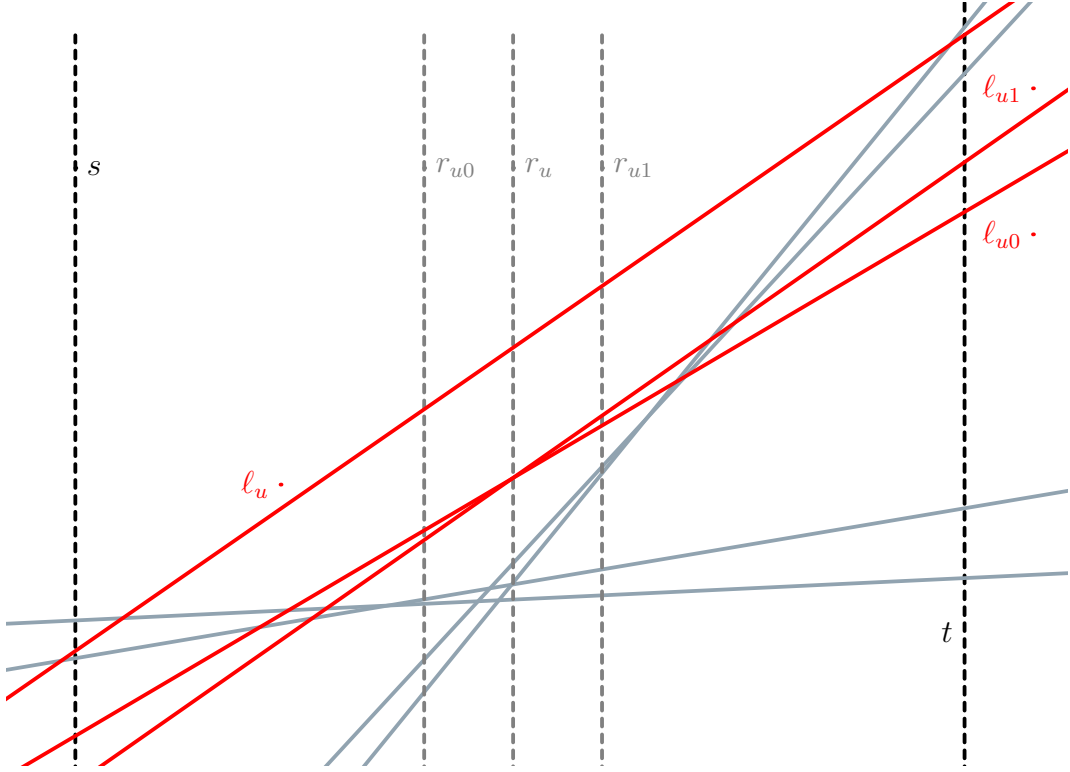
 By continuity, choosing $r_{u1} < 1/\mu = r_u < r_{u0}$ close enough to $r_u$ and within $(s,t)$ (see \autoref{summaryPic}), we can guarantee that
 \[\ell_{u0}(r_{u0})>\ell_v(r_{u0})\,\, \text{ and } \,\,\ell_{u1}(r_{u1})>\ell_v(r_{u1}).\]

Finally,  since $b_{u1}>b_{u0}$, the line $\ell_{u0}$ lies above the line $\ell_{u1}$ for values less than $1/\mu$, and conversely $\ell_{u0}$  lies below $\ell_{u1}$ for values greater than $\frac{1}{\mu}$. If $r_{u1} < 1/\mu < r_{u0}$, we have
\[
\ell_{u1}(r_{u1})> \ell_{u0}(r_{u1}),\,\, \text{ and } \,\, \ell_{u0}(r_{u0})>\ell_{u_1}(r_{u0}).
\]
Thus, our choice of parameters also satisfies condition (b3). It follows by induction that the property $P_u$ holds for all $u$. %
\end{proof}

\section{Discussion and open problems} \label{sec:dis}

We briefly discuss how \autoref{th:main} can be generalised to preorderd semirings with certain conditions (Strassen semirings), and then discuss some open problems.

\subsubsection*{Universality for Strassen semirings} 
Graphs with cohomomorphism are an instance of a so-called Strassen semiring. \autoref{th:main} generalizes under certain conditions that we will discuss here. 
 A semiring $(\mathcal{R},  0, 1, +, \cdot)$ with preorder $\leq_{\mathcal{R}}$ is called a Strassen semiring \cite{strassen1988asymptotic, wigderson2022asymptotic} if the following conditions are satisfied:
 \begin{enumerate}[\upshape(1)]
 \item The natural numbers in $\mathcal{R}$ (given by $0_\mathcal{R}$, $1_\mathcal{R}$, $ 2_\mathcal{R}= 1_\mathcal{R} +1_\mathcal{R}$, $3_\mathcal{R}=1_\mathcal{R} + 1_\mathcal{R} + 1_\mathcal{R}$, etc.) respect the natural ordering: for all $m,n \in \mathbb{N}$, $m \leq n$ if and only if $m \leq_{\mathcal{R}} n$.
 \item For every $a,b,c \in \mathcal{R}$, if $a \leq_{\mathcal{R}} b$, then $a +c \leq_{\mathcal{R}} b + c$ and $ac \leq_{\mathcal{R}} bc$.
 \item For all nonzero $a  \in \mathcal{R}$,  there exists an $n \in \mathbb{N}$ such that $1 \leq_{\mathcal{R}} a \leq_{\mathcal{R}} n.$
  \end{enumerate}
 We write
 $a  \lesssim_{\mathcal{R}} b$ for  $a,b\in \mathcal{R}$ if and only %
 $a^{n} \leq_{R} 2^{o(n)} b^{n}$. (In many applications, this condition is equivalent to  $a^{n} \leq_{\mathcal{R}}  b^{n + o(n)}$.) The asymptotic spectrum $\Delta(\mathcal{R})$ is the set of $\leq_{\mathcal{R}}$-monotone semiring homomorphisms %
 $f: \mathcal{R} \to \mathbb{R}_{\geq 0}$ 
 (just as in \autoref{sec:asd}). 
 Asymptotic spectrum duality \cite{strassen1988asymptotic} says that for any elements  $x, y  \in \mathcal{R}$, we have  $x\lesssim_\mathcal{R}  y$ if and only if for every $f \in \Delta(\mathcal{R})$, $f(x) \leq f(y)$.  %
 We need the following two conditions:
 \begin{enumerate}[(a)]
 \item There exists a binary operation $\star_{\mathcal{R}}$ such that for all spectral points $f$  of $\Delta(\mathcal{R})$ and $a,b \in \mathcal{R}$, $f(a \star_{\mathcal{R}}b) = \max(f(a),f(b))$. 
 \item There exists a family of elements $\{c_q\}_{q \in \mathbb{Q}_{\geq 2}}\subseteq \mathcal{R}$, a point $b \in \mathcal{R}$ and an  interval $[s,t]$ with $t>s>1$  such that for all $\lambda \in [s,t]$ there exists  a spectral point $f \in \Delta(\mathcal{R})$  such that  for all $q \in \mathbb{Q}$ and $\lambda \in [s,t]$ we have $f(c_q) =q$ and $f(b) = \lambda.$ Additionally, for all $q,r \in \mathbb{Q}_{\geq 2}$, if $q \leq r$, then $c_q \leq_{R} c_r.$
 \end{enumerate}
 \begin{theorem}\label{thm:StarRing}
    Let  $(\mathcal{R},  \leq_{\mathcal{R}})$ be a semiring equipped with a Strassen preorder %
    that satisfies the above properties {\upshape(a)} and {\upshape(b)}, then $(\mathcal{R},  \leq_{\mathcal{R}})$ and $(\mathcal{R}, \lesssim_\mathcal{R})$  are countably universal.
\end{theorem}
\begin{proof}[Proof sketch]
Suppose $(\mathcal{R},  \leq_{\mathcal{R}})$, $\{c_q\}_{q \in \mathbb{Q}}$, $[s,t]$ and $ b$ are as in the hypothesis of \autoref{thm:StarRing}. Let $\{(a_w,b_w,\ell_w) \}_{w \in W}$ be as in \autoref{aux} applied to $1<s<t$. Assign to each $w \in W$ the element $r_w := c_{a_w}b +c_{b_w} \in \mathcal{R}$ and to each $A \in \mathcal{W}$ the element $r_A: = \bigstar_{a \in A}\,r_a$. The proofs of \autoref{existenceOfNiceGa} and \autoref{th:main} can be adapted to this setting and yield an order-embedding from $(\mathcal{W}, \leq_{\mathcal{W}})$ to $(\mathcal{R},  \lesssim_\mathcal{R})$.
\end{proof}

\subsubsection*{Open problems}
We discuss several natural open problems and directions in the context of our universality results and the above general version.
\begin{itemize}
\item \textbf{Further instances, necessity of conditions.} What other Strassen semirings $(\mathcal{R},  \leq_{\mathcal{R}})$ satisfy properties (a) and (b)? (Then by \autoref{thm:StarRing} the preorder $\leq_{\mathcal{R}}$ and the asymptotic version $\lesssim_\mathcal{R}$ of $\leq_{\mathcal{R}}$ will be countably universal.)
It would be of interest to determine whether conditions (a) and (b) are necessary for a Strassen preorder to be universal, or if there is a weaker characterisation of such preorders.

\item \textbf{Universality of (asymptotic) tensor restriction.}
It would be particularly meaningful to know if properties (a) and (b) apply to the setting of tensors, the area in which the theory of asymptotic spectra initially emerged (motivated by the study of matrix multiplication algorithms and further connected to the study of quantum entanglement and several combinatorial problems like the cap set problem).
Here $\mathcal{R}$ is 
the semiring of tensors $T \in \mathbb{F}^{n_1} \otimes\mathbb{F}^{n_2} \otimes \mathbb{F}^{n_3}$ (for arbitrary~$n_i$) with sum given by the direct sum $\oplus$, product by the Kronecker product~$\boxtimes$, natural numbers given by identity tensors $\sum_{i=1}^r e_i \otimes e_i \otimes e_i$, and restriction order %
 $\leq_\mathcal{R}$ defined by $S \leq T $ if $S$ can be obtained from $T$ by applying linear maps to three factors.
Does the preorder $(\mathcal{R}, \leq_{\mathcal{R}})$ satisfy the conditions $(a)$ and $(b)$? Is this preorder countably universal?

\item \textbf{Polynomials with pointwise ordering and Kneser graphs.}
We have seen in \autoref{sec:sbr} how to use %
polynomials with coefficients in $\{0,1\} \cup \QQ_{\geq2}$, evaluated on a (proper) interval $[s,t]$, to order-embed
any finite preorder  %
into  $(\mathcal{G},\asympleq)$. Denote the set of such polynomials by $\mathcal{P}$ and for $p,q \in \mathcal{P}$, write $p \leq_{\mathcal{P}}q$ if and only if $p(x) \leq q(x)$  for all $x \in [s,t]$.  The strategy in \autoref{sec:sbr} worked because the preorder $(\mathcal{P}, \leq_{\mathcal{P}})$ is universal with respect to finite preorders. Based on this, we may imagine another way of proving \autoref{th:main}, namely (1) to prove $(\mathcal{P}, \leq_{\mathcal{P}})$ is universal for countable preorders; and then (2) to prove every $\phi \in \Delta(\mathcal{G})$ is ``Kneser-normalised'': 
 $\phi(\overline{K_{p:q}}) = p/q$ for every $p,q \in \mathbb{N}$ (where $K_{p:q}$ denotes the Kneser graph with vertices given by $q$-subsets of $[p]$ with adjacency given by disjointness). If all spectral points are Kneser-normalised, then  $\sum_i \frac{p_i}{q_i} x^i \to \sum_{i} \overline{K_{p_i:q_i}} G^{\boxtimes i}$ defines an order-embedding from $(P, \leq_P)$ to $(\mathcal{G}, \asympleq)$, and the image of this order-embedding is universal with respect to countable preorders.
     We leave the above (1) and (2) as open problems.
     Regarding problem (1), a weaker version of this statement is claimed in \cite[Theorem 2.4]{HubickaNestrilUniversal}, namely for polynomials with coefficients in~$\QQ$. %
     Regarding problem (2), known to be Kneser-normalised are: fractional clique covering number $\overline{\chi_f}$ \cite{Scheinerman1997FractionalGT}, Lovász theta function~$\theta$ \cite{Lovasz79} and complement of the projective rank $\overline{\xi_f}$ \cite{wocjan2018spectral}. 

\item \textbf{Fractal property.}
   Fiala, Hubička,  Long, and Nešetřil \cite{FractalProp}  prove that the cohomomorphism order %
   satisfies an even stronger property. For every two graphs $G_1$, $G_2$, if $G_1 \leq G_2$ and $G_2 \nleq G_1$, then there exists an order-embedding of $(\mathcal{G},\homomoleq) $ into  $(\{ H \in \mathcal{G} \mid G_1 \homomoleq H \homomoleq G_2\}, {\homomoleq})$, and thus the latter is universal. A preorder satisfying this property is called \emph{fractal}.
It is natural to ask 
whether $(\mathcal{G},\asympleq)$  is fractal. Our construction does not settle this.

\item \textbf{Computational complexity.}
From a computational complexity point of view, it is a central open problem to determine whether the Shannon capacity $\Theta(G)$ (and, closely related, asymptotic cohomomorphism $G \asympleq H$) is computable (see, e.g.,~\cite{MR2234473}). While \autoref{th:main} does not touch this problem, we expect that several of the ingredients we developed will be useful for the study of this and related complexity-theoretic questions.

\end{itemize}

\section*{Acknowledgements}
JZ and AL acknowledge financial support by NWO Veni grant VI.Veni.212.284 and NWO grant OCENW.M.21.316.

\bibliographystyle{alphaurl}
\bibliography{main}{}

\end{document}